\UseRawInputEncoding
\documentclass[12pt,reqno]{amsart}
\usepackage{amsmath,amsfonts,amsthm,amsopn,amssymb,extarrows,multirow,footnote,float}
\usepackage{cite,marginnote}
\usepackage{booktabs}
\usepackage{bm}
\usepackage{color,enumitem,graphicx}
\usepackage[colorlinks=true,urlcolor=blue,
citecolor=red,linkcolor=blue,linktocpage,pdfpagelabels,
bookmarksnumbered,bookmarksopen]{hyperref}
\usepackage[english]{babel}

\usepackage[left=2.9cm,right=2.9cm,top=2.8cm,bottom=2.8cm]{geometry}
\usepackage[hyperpageref]{backref}

\numberwithin{equation}{section}

\makeindex
\makeindex

\def\N{\mathbb{N}}

\newtheorem{theorem}{Theorem}[section]
\newtheorem{definition}[theorem]{Definition}
\newtheorem{lemma}[theorem]{Lemma}
\newtheorem{corollary}[theorem]{Corollary}
\newtheorem{proposition}[theorem]{Proposition}
\newtheorem{remark}[theorem]{Remark}

\newcommand{\s}{\section}

\newcommand{\R}{\mathbb R}

\newcommand{\lab}{\label}
\newcommand{\bt}{\begin{theorem}}
	\newcommand{\et}{\end{theorem}}
\newcommand{\bl}{\begin{lemma}}
	\newcommand{\el}{\end{lemma}}
\newcommand{\bd}{\begin{definition}}
	\newcommand{\ed}{\end{definition}}
\newcommand{\bc}{\begin{corollary}}
	\newcommand{\ec}{\end{corollary}}
\newcommand{\bp}{\begin{proof}}
	\newcommand{\ep}{\end{proof}}
\newcommand{\bx}{\begin{example}}
	\newcommand{\ex}{\end{example}}
\newcommand{\bi}{\begin{exercise}}
	\newcommand{\ei}{\end{exercise}}
\newcommand{\bo}{\begin{proposition}}
	\newcommand{\eo}{\end{proposition}}
\newcommand{\br}{\begin{remark}}
	\newcommand{\er}{\end{remark}}
\newcommand{\beq}{\begin{equation}}
	\newcommand{\eeq}{\end{equation}}
\newcommand{\ba}{\begin{align}}
	\newcommand{\ea}{\end{align}}
\newcommand{\bn}{\begin{enumerate}}
	\newcommand{\en}{\end{enumerate}}
\newcommand{\bg}{\begin{align*}}
	\newcommand{\bcs}{\begin{cases}}
		\newcommand{\ecs}{\end{cases}}

	\newcommand{\bean}{\begin{eqnarray*}}
		\newcommand{\eean}{\end{eqnarray*}}

	\def\N{\mathbb{N}}

	\def\R{\mathbb{R}}

	\def\bd{\mathrm{bd}\,}
	\definecolor{viola}{rgb}{1,0,1}

	\title[Uniqueness]{Uniqueness of radial solutions for $m$-Laplacian equations in low dimensions}
	
    \author[P.~Pucci]{Patrizia Pucci}
	\author[J.~J.~Zhang]{Jianjun Zhang}
	\author[X.~X.~Zhong]{Xuexiu Zhong}

\address[P.~Pucci]{\newline\indent Dipartimento di Matematica e Informatica
\newline\indent
Universit\`a degli Studi di Perugia
\newline\indent
via Vanvitelli 1, 06123 Perugia, Italy}
\email{\href{mailto:patrizia.pucci@unipg.it}{patrizia.pucci@unipg.it}}

\address[J.~J.~Zhang]{\newline\indent College of Mathematics and Statistics
\newline\indent
Chongqing Jiaotong University
\newline\indent
Xuefu, Nan'an, 400074, Chongqing, PR China}
\email{\href{mailto:zhangjianjun09@tsinghua.org.cn}{zhangjianjun09@tsinghua.org.cn}}

	\address[X.~X.~Zhong]{\newline\indent South China Research Center for Applied Mathematics and Interdisciplinary Studies \& School of Mathematical Sciences
		\newline\indent
		South China Normal University
		\newline\indent
		Guangzhou 510631, P. R. China}
	 \email{\href{mailto:zhongxuexiu1989@163.com}{zhongxuexiu1989@163.com}}

\date{\today}

\begin{document}

\begin{abstract}
    {This paper extends the uniqueness results of Serrin and Tang [\textit{Indiana Univ. Math. J.}, 49 (2000), pp. 897--923] to the low-dimensional case $1\leq N\leq m$ with $m>1$. We consider radial solutions of the overdetermined problem
    \[
    \begin{cases}
        -\Delta_m u = f(u), \quad u>0 & \text{in } B_R,\\[4pt]
        u = \partial_\nu u = 0 & \text{on } \partial B_R, \text{ if } R<\infty,\\[4pt]
        \displaystyle\lim_{|x|\to\infty} u(x)=0, & \text{if } R=\infty,
    \end{cases}
    \]
    where $B_R$ is the open ball in $\mathbb{R}^N$ centered at the origin with radius $R>0$ (the case $R=\infty$ corresponds to the whole space, for studying positive ground states). Under suitable assumptions on the nonlinearity $f$, we establish the uniqueness of such solutions, whenever they exist.}

    {Our analysis is motivated by connections to sharp forms of the Gagliardo--Nirenberg and Nash inequalities. Although the overall framework follows that of Serrin and Tang, the details of our proofs differ substantially in the low-dimensional setting. In particular, Serrin and Tang explicitly noted that their techniques rely heavily on the condition $N>m$ and do not readily extend to $N\leq m$ (see Subsection~6.2 of their work). The present paper closes this gap, thereby providing a complete uniqueness theory for all dimensions.}

    {As a concrete example, for the canonical nonlinearity $f(u) = -u^p + u^q$ with $p<q$, our result covers the full range $-1 < p < q < m^*-1$, where $m^*:=\frac{Nm}{N-m}$ for $N>m$ and $m^* = \infty$ for $N\leq m$. Consequently, our work also completely resolves an open problem posed by Pucci and Serrin [\textit{Indiana Univ. Math. J.}, 47 (1998), pp. 501--528],  which had been settled for $N>m$ in the earlier work of Serrin and Tang.}
\end{abstract}

\maketitle
	
\tableofcontents

\section{Introduction}
{For $R>0$, let $B_R$ denote the open ball in $\mathbb{R}^N$ centered at the origin with radius $R$ (when $R=\infty$, $B_R$ represents the entire space $\mathbb{R}^N$).
We consider radial solutions of the following overdetermined problem:
\begin{equation}\label{eq:20240919-1341}
\begin{cases}
-\Delta_m u = f(u), \quad u>0 & \text{in } B_R,\\[4pt]
u = \partial_\nu u = 0 & \text{on } \partial B_R, \text{ if } R<\infty,\\[4pt]
\displaystyle\lim_{|x|\to\infty} u(x)=0, & \text{if } R=\infty,
\end{cases}
\end{equation}
where $\Delta_m u = \operatorname{div}\bigl(|\nabla u|^{m-2}\nabla u\bigr)$ is the $m$-Laplace operator.
There are two mutually exclusive situations:
\begin{itemize}
    \item[(1)] \textbf{Regular case.} $f$ is continuous on $[0,\infty)$;
    \item[(2)] \textbf{Singular case.} $f$ cannot be extended to a continuous function on $[0,\infty)$.
\end{itemize}
For a solution $u$ of \eqref{eq:20240919-1341} with $R=\infty$, it is a positive ground state.
When $R<\infty$, we extend $u$ by zero outside $B_R$; then $u$ is a ground state if and only if $f$ is regular and can be continuously extended to $0$ by defining $f(0)=0$.
Here, a \textit{ground state} $u$ means that $u$ is a non-negative, non-trivial $C^1$ distributional solution of $-\Delta_m u = f(u)$ in $\mathbb{R}^N$ that tends to zero at infinity.}

When $1<m<N$, under the following assumptions:
{\it\begin{enumerate}[label=$\mathbf{(H_{\arabic*})}$]
\item $f$ is continuous on $(0,\infty)$, with $f\leq 0$ on $(0,b]$ and $f(u)>0$ for $u>b$, for some
$b>0$;\label{hcon_1}
\item $f\in C^1(b,\infty)$, with $g(u)=uf'(u)/ f(u)$ non-increasing on $(b,\infty)$,\label{hcon_2}
\end{enumerate}}
Serrin and Tang in \cite[Theorem 2]{SeT} proved that the Dirichlet-Neumanna free boundary problem \eqref{eq:20240919-1341} admits at most one radial solution. In particular, Serrin and Tang stated out clearly in \cite[Subsection~6.2]{SeT} that {\it``the proofs in the present paper rely extensively on the assumption $N>m$ and cannot be extended easily to values $N\leq m$."} This remark motivates us to extend their results to the  case $1\leq N\leq m$ and $m>1$.

Another {motivation} of this work is related to {\it the sharp Gagliardo-Nirenberg/Nash inequality}, which plays a crucial role in the study of partial differential equations (PDEs); in particular in the field of normalized solution problems.
\br\lab{remark:20241112-0953} {\rm
In \cite{Agu}, Martial Agueh established sharp Gagliardo-Nirenberg/Nash's inequalities:
\begin{gather}\lab{eq:20241112-0957}
\|u\|_{s}\leq K_{opt} \|\nabla u\|_{m}^{\theta} \|u\|_{q}^{1-\theta} \mbox{ for all $u\in Y^{m,q}(\R^N)$,
 where}\\
\theta=\dfrac{Nm(s-q)}{s[Nm-q(N-m)]},\; K_{opt}>0,\;
Y^{m,q}(\R^N):=\{u\in L^q(\R^N)\,:\, \nabla u\in L^m(\R^N)\},\nonumber
\end{gather}
and either
\begin{gather}\lab{eq:20241112-1118}
1<m<N,~\hbox{and}~1\leq q<s<m^*:=\frac{Nm}{N-m}
\quad\mbox{or}\\
\lab{eq:20241112-1119}
N=1,\; m>1,~\hbox{and}~1\leq q<s<\infty.
\end{gather}
The optimal constant $K_{opt}$ and optimal functions of the Gagliardo-Nirenberg/Nash inequality \eqref{eq:20241112-0957} are obtained by explicitly determining the minimum value of the variational problem:
\beq\lab{eq:20241112-1004}
\inf\left\{E(u):=\frac{1}{m}\|\nabla u\|_m^m +\frac{1}{q}\|u\|_q^q\,:\, u\in Y^{m,q}(\R^N),\; \|u\|_s=1\right\},
\eeq
which requires explicitly solving the $m$-Laplacian  equation:
\beq\lab{eq:20241112-1007}
-\Delta_m u+|u|^{q-2}u=\lambda |u|^{s-2}u,
\eeq
where $\lambda$ denotes the Lagrange multiplier for the constraint $\|u\|_s=1$. {The existence part can be established via variational principles. By employing a minimizing sequence, rearrangement inequalities, and the compact embedding theorem for radial functions, it is straightforward to show the compactness of the minimizing sequence. This procedure is standard. However, the uniqueness of this extremal function, or even the uniqueness of positive solutions to \eqref{eq:20241112-1007}, poses a nontrivial mathematical challenge. }
Martial Agueh\cite{Agu} applied the results of Serrin-Tang \cite{SeT} to show that $u_\infty$ is the unique radial, non-negative and non-trivial smooth solution of PDE \eqref{eq:20241112-1007}. Since the paper \cite{SeT} deals only with the case $1<m<N$, Martial Agueh also assumed $1<m<N$ in \eqref{eq:20241112-1118}.}
\er

 Take $m=2$ and $q=1$, as an example, and let $2<p<2\cdot 2^*$, $p=2s$. If $u\in L^1(\R^N)$ and $\nabla u\in L^2(\R^N)$,
 that is $u\in Y^{2,1}(\R^N)$, then the Gagliardo-Nirenberg/Nash inequality becomes
 \beq\lab{eq:20250107-1055}
\int_{\R^N}|u|^{\frac{p}{2}}\,{\rm d}x
\leq\frac{C(N, p)}{\|Q_p\|_{L^1(\R^N)}^{\frac{p-2}{N+2}}}
\left(\int_{\R^N} |\nabla u|^2\,{\rm d}x\right)^{\frac{N(p-2)}{2(N+2)}}
\left(\int_{\R^N} |u|\,{\rm d}x\right)^{\frac{4N-p(N-2)}{2(N+2)}},
\eeq
where $Q_p$ is the unique non-negative radial solution of the  Dirichlet-Neumann free boundary problem:
\beq\lab{eq:20250107-1054}
\begin{cases}
-\Delta u+1=u^{\frac{p}{2}-1},\ \ u>0~ &\mbox{in}~B_R,\\
u=\partial_\nu u=0~&\mbox{on}~\partial B_R.
\end{cases}
\eeq
{We emphasize that for $m=q=2$ and $2<s<2^*$ (where $2^*=\frac{2N}{N-2}$ for $N\geq 3$ and $2^*=\infty$ for $N=1,2$), Weinstein\cite{Weinstein-1983} first revealed a fundamental connection between finding a ground state of \eqref{eq:20241112-1007} and determining the best constant in the corresponding Gagliardo-Nirenberg inequality.}

\br\lab{Remark:20250107-1105} {\rm
Many papers cite \cite{Agu} and apply the sharp Gagliardo-Nirenberg/Nash inequality of \cite{Agu} in all dimensions. However, as already noted in
Remark~\ref{remark:20241112-0953}, when $1\leq N\leq m$ with $m>1$ requires an additional proof. {This is particularly relevant in situations requiring the uniqueness of the extremal function associated with the Gagliardo-Nirenberg/Nash inequality.}  Therefore, in order to apply the sharp Gagliardo-Nirenberg/Nash inequality rigorously and confidently in future in all dimension $N\geq 1$, we provide here the proof of  the uniqueness of $u_\infty$, when $1\leq N\leq m$ and $m>1$.}
\er

\br\lab{remark:20250107-1113} {\rm
In \cite{Jeanjean-Zhang-Zhong-2025}, Jeanjean, Zhang and Zhong considered the mass super critical case and proved the optimal existence of energy ground states, with the energy functional given by
$$ I(u) = \frac{1}{2}\int_{\R^N}|\nabla u|^2\mathrm{d}x + \int_{\R^N}u^2|\nabla u|^2\mathrm{d}x - \frac{1}{p}\int_{\R^N}|u|^p\mathrm{d}x$$
subject to the constraint in the set
$$\mathcal{S}_a=\{ u \in X\, :\, \int_{\R^N}| u|^2\mathrm{d}x = a\},$$
where $a>0$ is given, and
\begin{equation*}
X:=\left\{u\in H^1(\R^N)\,:\, \int_{\R^N}u^2|\nabla u|^2\mathrm{d}x<\infty\right\}.
\end{equation*}
They also captured the asymptotic behavior of these ground states as the mass $a \downarrow 0$, and $a \uparrow a^*$, where $a^*=\infty$ if
$1\leq N\leq 4$, and $a^*=a_0$ if $N\geq 5$. In particular, when $a \downarrow 0$ (that is when the corresponding Lagrange multiplier $\lambda_a\rightarrow \infty$), they discovered some new interesting phenomenon. Specifically, after appropriate transformations of the ground states, they converge to a function $\tilde{v}$ determined by the problem
\beq\lab{eq:20240905-1912}
\begin{cases}
-\Delta \tilde{v}=-\frac{\sqrt{2}}{2}+2^{\frac{p-4}{4}}
\tilde{v}^{\frac{p-2}{2}},~ \tilde{v}>0~&\hbox{in}~B_R,\\
\tilde{v}=\partial_\nu \tilde{v}=0~&\hbox{on}~\partial B_R.
\end{cases}
\eeq
Our result in the present paper provides a theoretical support for the uniqueness, particularly in the one-dimensional case.}
\er

In \cite[Theorem 3]{Pucci-Serrin1998}, Pucci and Serrin considered a more general class of quasilinear elliptic operators, including the $m$-Laplace operator, and their results also cover the case $N=m$. Specifically, for $1<m<\infty$, $2\leq N<\infty$, they proved
that the Dirichlet-Neumanna free boundary problem \eqref{eq:20240919-1341} admits at most one radial ground state, provided that $f$ satisfies the following assumptions:
{\it\begin{enumerate}[label=$\mathbf{(H'_{\arabic*})}$]
\item  \label{H-p-1}$f$ is continuously differentiable on $(0,\infty)$;
\item  \label{H-p-2}there exists $b>0$ such that $f(b)=0$ and
$$
f(u)<0~\hbox{for}~0<u<{b},\quad
f(u)>0~\hbox{for}~{b}<u<\infty;
$$
\item  \label{H-p-3}$f$ is locally integrable on $[0,\infty)$. In particular, the integral $F(u)=\int_0^u f(\tau)\mathrm{d}\tau$ then exists, and $F(u)\rightarrow 0$ as $u\rightarrow 0+$ (note that this applies even if $f(u)\rightarrow -\infty$ as $u\rightarrow 0+$);
\item  \label{H-p-4}$\displaystyle \frac{\mathrm{d}}{\mathrm{d}u}\left(\frac{F(u)}{f(u)}\right)\geq \frac{N-m}{Nm}$ for $u>0, u\neq b$.
\end{enumerate}}

 In \cite{Tang-2001} Tang fully clarified uniqueness and global structure of positive radial solutions for
  the $m$-Laplace equation. Adachi and Shibata in \cite{Adachi2018} also analyzed uniqueness and the non-degeneracy of positive radial solutions for semi-linear elliptic problems. However, the authors of \cite{Tang-2001} and \cite{Adachi2018} only considered the case $N\geq 2$ and assumed that $f(0)=0$, which is not required in our assumptions. When $N=1$, Berestycki and Lions in \cite[Theorem~5]{BL} provided necessary and sufficient conditions for the existence of solution for a certain class of semi-linear elliptic problems. However, they covered only the case $f(0)=0$, with  proof techniques not applicable when $f(0)<0$.

Based on the reasons outlined in Remarks~\ref{remark:20241112-0953}--\ref{remark:20250107-1113}, we extend here some theorems in~\cite{SeT} to the case when $1\le N\leq m$ and $m>1$. {Specifically, we introduce the following additional condition:
{\it\begin{enumerate}[label=$\mathbf{(H_{3})}$]
\item When $2\leq N\leq m$, $g(u)>-1$ for $u>b$.\label{hcon_3}
\end{enumerate}}
Then our main result can be stated as follows:
\bt\lab{th:20250107-1127}
Let $m>1$, $1\leq N\leq m$ and let \ref{hcon_1} hold.
\begin{itemize}
\item[(i)] If $N=1$ and \eqref{eq:20240919-1341} admits a radial solution $u$, then $u$ is unique and $u(0)>b$ is uniquely determined by $F(u(0))=0$.
\item[(ii)] Assume further that \ref{hcon_2}-\ref{hcon_3} hold if $N\geq 2$. Then the free boundary problem \eqref{eq:20240919-1341} admits at most one radial solution.
\end{itemize}
\et
}

\begin{remark}\label{remark:20260123-1341}
    \rm{
        The function $g(u) := u f'(u)/f(u)$ is known as the \textit{order function} of $f$.
        The additional condition \ref{hcon_3} is mild and easily verifiable.
        Since \ref{hcon_2} implies that $g$ is non-increasing on $(b, \infty)$, it suffices to determine the asymptotic order of $f$ at infinity.
        In particular, if $f$ has an order greater than $-1$ at infinity, then condition \ref{hcon_3} is satisfied.
    }
\end{remark}

{A key example is the canonical nonlinearity $f(u) = -u^p + u^q$ with $p < q$.
To ensure the existence of radial solutions for \eqref{eq:20240919-1341}, it is natural to impose the further condition (see \cite[condition (H3)]{SeT}):
$$
\int_0 |f(s)| \, \mathrm{d}s < \infty \quad \text{and} \quad F(u) = \int_0^u f(s) \, \mathrm{d}s < 0 \quad \text{for small } u,
$$
which requires $p>-1$.}

{Let $m^* = \frac{mN}{N-m}$ for $N > m$ and $m^* = \infty$ for $N \leq m$.
In the Laplacian case (i.e., $m = 2$), Coffman \cite{Coffman-1972} studied the uniqueness of ground states for \eqref{eq:20240919-1341} when $p = 1$, $q = 3$, and $N = 3$.
Mcleod and Serrin \cite{Mcleod-Serrin-1987} later extended Coffman's approach to obtain more general results.
For $1=p<q$, their results apply to all $q<\infty$ when $N=1,2$; however, for $3 \leq N \leq 7$, they only cover a partial range of $(1, 2^*-1)$, and did not address the case $N \geq 8$.
The entire range $1=p<q<2^*-1$ was subsequently resolved by Kwong \cite{Kwong-1989} for $N\geq 3$.}

{These techniques have been generalized to quasilinear problems; see \cite{Franchi-Lanconelli-Serrin-1996, Citti-1993, Pucci-Serrin1998} and references therein.
For the canonical nonlinearity $f(u) = -u^p + u^q$, uniqueness for \eqref{eq:20240919-1341} was established by Citti \cite{Citti-1993} for the range
$$
1 < m \leq 2, \quad m - 1 \leq p \leq 1, \quad m - 1 < q < m^* - 1.
$$
Additional ranges of $p$ and $q$ were obtained by Pucci and Serrin \cite{Pucci-Serrin1998} (even for more general operators).
Specifically, a partial range of $0 < p < q < m^* - 1$ was covered in \cite[Theorems 2 and 2']{Pucci-Serrin1998}, and later extended to part of $-1 < p < q < m^* - 1$ in \cite[Theorems 4 and 4']{Pucci-Serrin1998}.
This led Pucci and Serrin to pose the open question: \textit{whether uniqueness holds for the entire range $-1 < p < q < m^* - 1$}. The key point is that condition \ref{H-p-4} is not hold for $f(u)=-u^p+u^q$ in the entire range $-1 < p < q < m^* - 1$.}

{This question was resolved positively by \cite[Page 899, Corollary]{SeT} for $N > m$.
Note that $f$ has order $q$ at infinity, and since $q > p > -1$, condition \ref{hcon_3} is satisfied.
Therefore, our Theorem \ref{th:20250107-1127} also resolves this open question for $1 \leq N \leq m$ and $m > 1$.}

{
\begin{remark}\label{remark:20260124-1416}
 {\rm   A sign-changing radial solution of
    \begin{equation}\label{eq:20260124-1418}
        -\Delta u = f(u) \quad \text{in } \mathbb{R}^N, \quad u \in H^1(\mathbb{R}^N), \quad u \not\equiv 0,
    \end{equation}
    is called a \textit{bound state} (or \textit{excited state}).
    For $2\leq N\leq 4$, the uniqueness of such a solution having exactly $k$ zeros for $|x| > 0$ (referred to as the $k$-th bound state) was established by Cort\'azar, Garc\'ia-Huidobro, and Yarur \cite{CGHY-2011} under some suitable assumptions.
    In particular, for $f(u)=-u+|u|^{q-1}u$ with $1<q< 2^*-1$, their work establishes the uniqueness of the $1$-th bound state for \eqref{eq:20260124-1418} in dimensions two. For $N=q=3$ and $k\leq 20$, Cohen, Li and Schlag \cite{Cohen-Li-Schlag-2024} proved the uniqueness  of $k$-the bound state.
    Recently, for $f(u)=-u+|u|^{q-1}u$ with entire range $1<q< 2^*-1$ and all $k\in \mathbb{N}$, Tang \cite{Tang-2026} completely resolved the uniqueness of the $k$-th bound state for all dimensions $N\geq 3$, which gave a positive answer to the long-standing conjecture of Berestycki and Lions\cite{BL}.}
\end{remark}
}

\section{{Radial Solutions: Monotonicity and Boundary Behavior}}
The results established in this section are based only on assumption \ref{hcon_1}. The radial solutions of~\eqref{eq:20240919-1341} satisfies the equation in the form
{
\beq\lab{eq:20240919-1424}
\begin{cases}
(r^{N-1}u'|u'|^{m-2})'=-r^{N-1}f(u),\quad r\in [0,R),\\
u>0~\hbox{in}~[0,R),\quad u'(0)=0,\\
u(R)=u'(R)=0~\hbox{if}~R<\infty~\hbox{and}~\lim_{r\rightarrow \infty}u=0~\hbox{if}~R=\infty,
\end{cases}
\;\phantom{a}'=\frac{\mathrm{d}}{\mathrm{d}r}.
\eeq
}
Let us denote by $J=(0,R)$, {$R\in (0,\infty]$},  the maximal existence interval of $u$ on which $u>0$.

\bo\lab{prop:20240919-1425}Let $m>1$ and $1\leq N\leq m$.  It holds that
$\alpha=u(0)>b$ and $u'(r)<0$ as  long as $u(r)>0$ and $r>0$.
\eo
\bp
A result of this type when {$1<m<N$} was established in \cite[Proposition 2.1]{SeT}. However, the proof when $1\leq N\leq m$ is different.

Suppose that $u'(r_0)=0$ for some {$r_0\in[0,R)$} at which $u(r_0)>0$, {we state the following claim:}\\
 {\bf Claim 1:} $u_0:=u(r_0)>b$.

{\bf Case $N=1$:} Multiplying \eqref{eq:20240919-1424} by $u'$ and  integrating from $r_0$ to $R$,
{$$-\frac{m-1}{m}|u'(r)|^m\Big|_{r_0}^{R}=-\int_{r_0}^{R} (m-1) |u'|^{m-2} u'' u' \mathrm{d}r =\int_{r_0}^{R} f(u)u'\mathrm{d}r.$$
Since} $u'(r_0)=0$ and
$$\begin{cases}
u(R)=u'(R)=0,\quad ~&\hbox{if}~R<\infty,\\
\lim\limits_{r\rightarrow \infty}u(r)=\lim\limits_{r\rightarrow \infty}u'(r)=0,&\hbox{if}~R=\infty,
\end{cases}$$
we obtain that $F(u(r_0))=0$.
Let $r^*$ be the largest critical point of $u$ in $[0,R)$ (see Remark \ref{remark:20241119-1119} below). That is, $u'(r)\neq 0$ for all $r\in (r^*, R)$. Then, again $F(u(r^*))=0$. If $u(r^*)\leq b$, then $f(u(r))\leq 0$ for $r\in [r^*, R)$. However, $f(u(r))\not\equiv 0$ in $[r^*, R)$. If not, $u(r)\equiv u(r^*)$ on $[r^*, R)$, which contradicts the fact that $u(r)\rightarrow 0$ as $r\rightarrow R$. Thus, $\int_{r^*}^{R}f(u(r))\mathrm{d}r<0$. On the other hand,
$$0=u'(r^*)|u'(r^*)|^{m-2}=-\int_{r^*}^{R}(u' |u'|^{m-2})'\mathrm{d}r=\int_{r^*}^{R}f(u(r))\mathrm{d}r,$$
which is impossible. Therefore, $u(r^*)> b$.

Suppose that  $u(r_0)\leq b$. This and the facts that $u(r^*)>b$ and $u(r)\rightarrow 0$ as $r\rightarrow R$ imply the existence of some $\bar{r}_0\in (r^*, R)$ such that $u(\bar{r}_0)=u(r_0)$.
We define
\beq\lab{eq:20241119-1031}
\rho(r):=\frac{m-1}{m}|u'(r)|^m +\int_{u(r_0)}^{u(r)} f(s)\mathrm{d}s, ~r\in [r_0, \bar{r}_0].
\eeq
Then $\rho(r_0)=0$ and $\rho(\bar{r}_0)=\frac{m-1}{m}|u'(\bar{r}_0)|^m\geq 0$.
Now, in $[r_0, \bar{r}_0]$
\begin{align*}
\rho'(r)=&(m-1)|u'(r)|^{m-2}u'(r)u''(r)+f(u(r))u'(r)\\
=&u'(r)[(m-1)|u'(r)|^{m-2}u''(r)-\frac{N-1}{r}|u'(r)|^{m-2}u'(r)
-(m-1)|u'(r)|^{m-2}u''(r)],
\end{align*}
that is
\beq\lab{eq:20241119-1033}
\rho'(r)=-\frac{N-1}{r}|u'(r)|^m\leq 0.
\eeq
Therefore, $\rho'\equiv 0$ in $[r_0, \bar{r}_0]$, since $N=1$. In particular,  $\rho(r)\equiv \rho(r_0)$ in $[r_0, \bar{r}_0]$. Thus,
$\frac{m-1}{m}|u'(\bar{r}_0)|^m=\rho(\bar{r}_0)=\rho(r_0)=0$, which implies that $u'(\bar{r}_0)=0$, with $u(\bar{r}_0)>0$. This  contradicts the definition of $r^*$ and completes the proof of the Case $N=1$.

{\bf Case $N\geq 2$:} {To complete the proof of \textbf{Claim 1}, we now formulate our second claim.}

{\textbf{Claim 2:} $u(r)\leq u(r_0)$ for all $r>r_0$.} \\
 If not, there exists some $r_1>r_0$ such that $u(r_1)>u(r_0)>0$. Since $u(r)$ vanishes at $r=R$ or infinity, one can find some $\bar{r}_0>r_1$ such that $u(\bar{r}_0)=u(r_0)$. Let $\rho$ be the function defined by \eqref{eq:20241119-1031} above. Then by \eqref{eq:20241119-1033}, $\rho$ is non-increasing in $[r_0, \bar{r}_0]$. In particular, $u'\not\equiv 0$ in $[r_0, \bar{r}_0]$. Thus,
$$\rho(\bar{r}_0)=\rho(r_0)+\int_{r_0}^{\bar{r}_0}\rho'(r)\mathrm{d}r=-(N-1)\int_{r_0}^{\bar{r}_0} \frac{1}{r}|u'(r)|^m\mathrm{d}r<0.$$
This contradicts that $\rho(\bar{r}_0)=\frac{m-1}{m}|u'(\bar{r}_0)|^m\geq 0$ and proves {\bf Claim 2}.

In conclusion, if $0<u_0\leq b$, then {\bf Claim 2} shows that $u(r)\leq u_0\leq b$ for all $r>r_0$. Hence, $f(u(s))\leq 0$ for $s\in [r_0, r]$. Now, \eqref{eq:20240919-1424} and $u'(r_0)=0$ give
\beq\lab{eq:20240919-1457}
r^{N-1}u'(r)|u'(r)|^{m-2}=-\int_{r_0}^{r}s^{N-1}f(u(s))\mathrm{d}s.
\eeq
Thus, $u'(r)\geq 0$ for all $r>r_0$, which is clearly
impossible, since $u(r_0)>0$, with $u(r)\rightarrow 0$ as $r\rightarrow R\leq \infty$. Hence, {\bf Claim 1} is proved.
\vskip 0.2in

Since $u(0)=\alpha$ and $u'(0)=0$, it follows (by {\bf Claim 1}) as well that $\alpha>b$. By $u(r)\rightarrow 0$ as $r\rightarrow R\leq \infty$ and $u(0)=\alpha>b$, the value
\beq\lab{eq:20240919-1757}
r_2:=\sup\{s>0: u(r)>b~\hbox{for all}~r\in (0,s)\}\in (0,\infty)
\eeq
is well defined and clearly
$u(r_2)=b$. Integrating \eqref{eq:20240919-1424} from $0$ to $r$, we obtain that
\beq\lab{eq:20240919-1758}
r^{N-1}u'(r)|u'(r)|^{m-2}=-\int_0^r s^{N-1}f(u(s))\mathrm{d}s,
\eeq
since $u'(0)=0$.
The definition of $r_2$ yields that $u(r)> b$ for any $r\in (0, r_2)$. Hence,  \ref{hcon_1} gives that $f(u(r))>0$ for $r\in (0, r_2)$. Then, \eqref{eq:20240919-1758} implies that $u'<0$ in $(0, r_2)$.

{\bf Claim 1} again provides that $u'(r)\neq 0$ when $0<u(r)\leq b$. Therefore, $u'(r_2)<0$ and so $u'(r)<0$ as long as $0<u(r)<b$. In summary, $u'(r)<0$ as long as $u(r)>0$ and $r>0$. This completes the proof.
\ep

\br\lab{remark:20241119-1119} {\rm
For $N=1$, we note that $r^*$ is well defined. If not, there exists a sequence $(r_n)_n\subset [0,R)$ such that $r_n\uparrow R, u'(r_n)=0$. Since $u(r_n)>0$ and $u(r)\rightarrow 0$ as $r\rightarrow R$, there exists some $i\in \N$ such that $u(r)\in (0,b)$ for $r\in [r_i, R)$.
One can see that $f(u(r))\not\equiv 0$ on $[r_i,R)$. If not, $(u'|u'|^{m-2})'\equiv 0$ on $[r_i,R)$, which, combining with $u'(R)=0$, leads to $u'\equiv 0$. Thus $u\equiv const$ in $[r_i,R)$, which is impossible. Hence, there exists $j\geq i$ such that $u(r_j)>u(r_{j+1})$. By the same argument, it is easy to see that $f(u(r))\not\equiv 0$ on $[r_j, r_{j+1}]$.
Put
$$\rho(r):=\frac{m-1}{m}|u'(r)|^m +\int_{u(r_j)}^{u(r)} f(s)\mathrm{d}s, ~r\in [r_j, r_{j+1}].$$
Then, similar to \eqref{eq:20241119-1033}, we conclude that $\rho'\equiv 0$ in $[r_j, r_{j+1}]$ and so $\rho(r_{j+1})=\rho(r_{j})=0$.
However,
$$\rho(r_{j+1})=\frac{m-1}{m}|u'(r_{j+1})|^m +\int_{u(r_j)}^{u(r_{j+1})} f(s)\mathrm{d}s
=\int_{u(r_j)}^{u(r_{j+1})} f(s)\mathrm{d}s>0,$$
which is the desired contradiction.}
\er

\bo\lab{prop:20240919-1859} Let $m>1$ and  $N\geq 1$. Then
$\displaystyle \frac{|u'(r)|^{m-1}}{r}\rightarrow \frac{f({u(0)})}{N}$ as $r\downarrow 0$.
\eo
\bp
Indeed, the proof of \cite[Proposition 2.3]{SeT} is  valid for all $N\geq 1$ (not necessary $N>m$). Thus, we omit it here.
\ep

\s{An equivalent problem and some properties}
The results establish in this section are also based only on assumption \ref{hcon_1}.
Proposition~\ref{prop:20240919-1425} yields that $u'<0$ in $J=(0,R)$.
 In particular, $R=\infty$ if $u$ is positive in the whole space $\R^N$. Otherwise, $R<\infty$.

Since $u'<0$ in $J=(0,R)$, it follows that the inverse of $u=u(r)$, denoted by $r=r(u)$, is well defined and is strictly decreasing on $(0,\alpha)$. By \eqref{eq:20240919-1424}, we conclude that
\beq\lab{eq:20240919-1817}
(m-1)r''=\frac{N-1}{r}{r'}^2+f(u)|r'|^m r',~u\in (0,\alpha), \quad '=\frac{\mathrm{d}}{\mathrm{d}u}.
\eeq
{We observe that the inversion method was initially developed by Peletier and Serrin \cite{Peletier-Serrin-1983} in the context of the Laplacian problem. This approach has played a significant role in establishing the uniqueness of ground states and admits a natural extension to quasilinear elliptic equations; see \cite{Pucci-Serrin1998,SeT} and references therein. Further technical developments and related discussions in this area can be found in \cite{Tang-2026}.}

\bo\lab{prop:20240919-1819}Let $m>1$ and $1\leq N\leq m$. If $u$ is a solution of \eqref{eq:20240919-1424}, then
\begin{itemize}
       \item[$(i)$]
$\displaystyle{\lim_{u\rightarrow 0}\frac{r^{\frac{N-1}{m-1}}(u)}{|r'(u)|}=\lim_{r\rightarrow R}r^{\frac{N-1}{m-1}}|u'(r)|=0}$ and
\item[$(ii)$]
$\displaystyle{\lim_{u\rightarrow 0}ur^{\frac{N-m}{m-1}}(u)=\lim_{r\rightarrow R}r^{\frac{N-m}{m-1}}u(r)=0.}$
 \end{itemize}
\eo
\bp
Let us first prove $(ii)$. If $N=m$, the conclusion is trivial. If $1\le N<m$, with $R=\infty$, then $r^{\frac{N-m}{m-1}}u(r)\leq \alpha r^{\frac{N-m}{m-1}}\rightarrow 0$ as $r\rightarrow \infty$. If $1\le N<m$, with $R<\infty$, then $u(R)=0$ and $\displaystyle{\lim_{r\rightarrow R}r^{\frac{N-m}{m-1}}u(r)=R^{\frac{N-m}{m-1}}} u(R)=0$.

We now turn to the proof of $(i)$. If $R<\infty$, then $(i)$ follows from $u'(R)=0$ and so $R^{\frac{N-m}{m-1}}u(R)=0$. If $R=\infty$ and $N=1$, then $u'(r)\rightarrow 0$ as $r\rightarrow \infty$, which implies $(i)$. While, if $R=\infty$ with $2\leq N\leq m$, recalling that $(r^{N-1}u'|u'|^{m-2})'=-r^{N-1}f(u)\geq 0$ for $r>r_2$ with $u(r_2)=b$, we see that $r^{N-1}u'|u'|^{m-2}<0$ is nondecreasing on $(r_2,\infty)$. Then there exists some finite $\mu\geq 0$ such that $\lim\limits_{r\rightarrow \infty}r^{N-1}u'|u'|^{m-2}=-\mu$. We {\bf claim} that $\mu=0$. If not, $r^{N-1}u'|u'|^{m-2}\leq -\mu$ for all $r>r_2$. Thus, for $r>r_2$, it holds that $u'(r)\leq {-}\mu^{\frac{1}{m-1}}r^{-\frac{N-1}{m-1}}$ and
$$-\alpha<-u(r)=\int_r^\infty u'(s)\mathrm{d}s<-\mu^{\frac{1}{m-1}} \int_r^\infty s^{-\frac{N-1}{m-1}}\mathrm{d}s=-\infty ~\hbox{being $\frac{N-1}{m-1}\leq 1$}.$$
This contradiction shows that $\mu=0$ and concludes the proof of~$(i)$.
\ep

\bo\lab{prop:20240919-2007}
Let $m>1$ and $1\leq N\leq m$. Suppose that $u_1$ and $u_2$ are two distinct radial solutions to Eq.~\eqref{eq:20240919-1424}. Then, there is no $\bar{r}>0$ such that
\beq\lab{eq:20240919-2008}
u_1(\bar{r})=u_2(\bar{r})>0~\hbox{and}~u'_1(\bar{r})=u'_2(\bar{r}).
\eeq
In particular, $\alpha_1\neq \alpha_2$, where $\alpha_i:=u_i(0)$ for $i=1,2$.
\eo
\bp
 Both inverses $r_1(u)$, $r_2(u)$  satisfy \eqref{eq:20240919-1817} for $0<u<\min\{\alpha_1,\alpha_2\}$.  The RHS of \eqref{eq:20240919-1817} is locally Lipschitz continuous {in the dependent variables $r,r'$ (for $r>0, r'<0$)}. Now, if \eqref{eq:20240919-2008}
 holds at some $\bar{r}>0$, then
\beq\lab{eq:20240919-2014}
r_1(u)=r_2(u)=\bar{r}>0,\quad r'_1(u)=r'_2(u)
\eeq
at the ``initial" value $\bar{u}=u_1(\bar{r})=u_2(\bar{r})>0$. Hence, $r_1(u)\equiv r_2(u)$ for $0<u<\min\{\alpha_1, \alpha_2\}$ and in turn $\alpha_1=\alpha_2$. It follows that the corresponding intervals $J_1$ and $J_2$ for $u_1$ and $u_2$ are also the same, which contradicts the assumption that $u_1\not\equiv u_2$.

{The case of $\bar{r}=0$ requires separate consideration. Suppose that $\alpha_1=\alpha_2=:\alpha$. By Proposition \ref{prop:20240919-1425}, $\alpha>b$. Thus, for some small $\delta>0$, we have $u_1(r), u_2(r)>b$ on $[0,\delta)$. Although $f$ is only  continuous (not necessarily Lipschitz) on $(0,b)$,  assumption \ref{hcon_2} ensures Lipchitz continuous on $(b,\infty)$. Consequently, both $u_1$ and $u_2$ solve the initial value problem
\beq\lab{eq:20240919-2112}
\begin{cases}
(r^{N-1}u'|u'|^{m-2})'+r^{N-1}f(u)=0, \quad 0<r<\delta~\hbox{for some $\delta>0$},\\
u(0)=\alpha,\quad
u'(0)=0.
\end{cases}
\eeq
It follows that $u_1(r)\equiv u_2(r)$ on $[0,\delta)$. By earlier result for $\bar{r}>0$, we obtain the global identity $u_1\equiv u_2$ and $J_1=J_2$, contradicting the hypothesis $u_1\not\equiv u_2$.}
\ep

\s{Proof of Theorem \ref{th:20250107-1127} for $N=1$}
Now, we are ready to establish Theorem \ref{th:20250107-1127}
when $N=1$ under the solely assumption \ref{hcon_1}.

Suppose for contradiction that $u_1$ and $u_2$ are two distinct radial solutions. By Propositions~\ref{prop:20240919-1425} and~\ref{prop:20240919-2007}, without loss of generality, we assume that $b<\alpha_1<\alpha_2$. Let $r_1, r_2$ be the inverse of $u_1, u_2$, respectively. Put
\beq\lab{eq:20240919-2033}
B(u):=\frac{1}{|r'_1|^m}-\frac{1}{|r'_2|^m},\quad u\in(0,\alpha_1).
\eeq
Since $N=1$ and $r'(u)<0$, recalling \eqref{eq:20240919-1817}, we have
$$r''_i=\frac{1}{m-1}f(u){|r'_i|}^m r'_i={-\frac{1}{m-1}}f(u)|r'_i|^{m+1}.$$
A direct computation, combined with \eqref{eq:20240919-1817}, shows that
\begin{align*}
B'(u)=&-m \frac{r'_1 r''_1}{|r'_1|^{m+2}}+m \frac{r'_2 r''_2}{|r'_2|^{m+2}}\\
=&\frac{m r''_1(u)}{|r'_1(u)|^{m+1}} -\frac{m r''_2(u)}{|r'_2(u)|^{m+1}}\\
=&-{\frac{m}{m-1}}f(u)+{\frac{m}{m-1}}f(u)=0.
\end{align*}
Hence,
\beq\lab{eq:20240919-2045}
B\equiv \text{const},\;\;\mbox{ in }(0, \alpha_1).
\eeq
Letting $u\rightarrow 0^+$, we obtain that
$$B(u)\equiv \lim_{u\rightarrow 0^+}\left[\frac{1}{|r'_1|^m}-\frac{1}{|r'_2|^m}\right]=0.$$
This implies that $r'_1\equiv r'_2$ in  $(0,\alpha_1)$.
Thus,
\beq\lab{eq:20240919-2050}
r_1(u)-r_2(u)\equiv r_1(\alpha_1)-r_2(\alpha_1)=-r_2(\alpha_1)<0, \quad 0<u<\alpha_1.
\eeq
Hence, $J_1\subsetneqq J_2$, i.e., $R_1<R_2$. In particular, $u_2(r+r_2(\alpha_1))\equiv u_1(r)$ on $J_1$. Hence, $u'_1(0)=u'_2(r_2(\alpha_1))<0$. Thuis contradicts the fact that $u'_1(0)=0$.

{The equality $F(u(0))=0$ also holds, as established in the proof of Proposition \ref{prop:20240919-1425}.}
\hfill$\Box$

\br\lab{remark:20240924-1806}
{\rm Actually, the case $N=1$ can also be proved in another way.
Suppose by contradiction that $u_1$ and $u_2$ are two distinct radial solutions. By Propositions~\ref{prop:20240919-1425}
and~\ref{prop:20240919-2007}, we may assume that $b<\alpha_1<\alpha_2$. Then,  {$F(\alpha_i)=F(u_i(0))=0$, $i=1,2$,} which is impossible, since $f(u)>0$ for $u\in (\alpha_1, \alpha_2)$.}
\er

\s{Proof of Theorem \ref{th:20250107-1127} for $2\leq N\leq m$}
\subsection{{Two auxiliary functions for $2\leq N\leq m$: $A(u)$ and $P(u)$}}
Let us first introduce the function
\begin{equation}\label{A}
A(u)=(N-m)u+(m-1)\frac{r(u)}{r'(u)}, \quad u\in (0, \alpha).
\end{equation}
\begin{lemma}\label{lemma:20240919-1907}
    Let $2 \leq N \leq m$. Then $A(b) \leq 0$ and $A(u) \rightarrow 0$ as $u \downarrow 0$.
\end{lemma}

\begin{proof}
   A direct computation yields
    \begin{equation}\label{eq:20240919-1922}
        \begin{aligned}
            A'(u) &= (N - m) + (m - 1) \frac{{r'}^2 - r r''}{{r'}^2} \\
                  &= N - 1 - \frac{r}{{r'}^2} (m - 1) r'' \\
                  &= N - 1 - \frac{r}{{r'}^2} \left[ \frac{N - 1}{r} {r'}^2 + f(u) |r'|^m r' \right] \quad \text{by \eqref{eq:20240919-1817}} \\
                  &= -f(u) r |r'|^{m - 2} r'.
        \end{aligned}
    \end{equation}
    Since $r(u) > 0$, $r'(u) < 0$, and $f(u) \leq 0$ for $u \in (0, b)$, it follows that $A'(u) \leq 0$ on $(0, b)$. Hence, $A$ is non-increasing on $(0, b)$ and the limit
    $$
    A(0) := \lim_{u \downarrow 0} A(u)
    $$
    is well defined.

    We now prove that $A(0) = 0$. By direct computation,
    \begin{equation}\label{eq:20250107-1848}
        \begin{aligned}
            A(0):=\lim_{u \downarrow 0} A(u) &= \lim_{u \downarrow 0} \left[ (N - m) u + (m - 1) \frac{r(u)}{r'(u)} \right] \\
            &= (m - 1) \lim_{u \downarrow 0} \frac{r(u)}{r'(u)} \\
            &= (m - 1) \lim_{u \downarrow 0} \left( \frac{r^{\frac{N - 1}{m - 1}}(u)}{r'(u)} \cdot r^{\frac{m - N}{m - 1}}(u) \right).
        \end{aligned}
    \end{equation}
    We consider the following three cases.

    \textbf{Case 1: $N = m$.}
    By Proposition~\ref{prop:20240919-1819}~(i), $\frac{r(u)}{|r'(u)|} \rightarrow 0$ as $u \downarrow 0$, so $A(u) \rightarrow 0$ as $u \downarrow 0$.

    \textbf{Case 2: $N < m$ and $R < \infty$.}
    Then $r^{\frac{m - N}{m - 1}}(u) \rightarrow R^{\frac{m - N}{m - 1}}$ as $u \downarrow 0$. By Proposition~\ref{prop:20240919-1819}~(i) again,
    \begin{equation}\label{eq:20250107-1854}
        \lim_{u \downarrow 0} A(u) = (m - 1)R^{\frac{m - N}{m - 1}}\lim_{u \downarrow 0}\frac{r^{\frac{N - 1}{m - 1}}(u)}{r'(u)} = 0.
    \end{equation}

    \textbf{Case 3: $N < m$ and $R = \infty$.} Recalling \eqref{eq:20250107-1848}, let
    \begin{equation}\label{eq:20251119-1950}
        \sigma := \lim_{u \downarrow 0} \frac{r(u)}{r'(u)} = \lim_{r \rightarrow \infty} u'(r) r=\frac{1}{m-1}A(0)
    \end{equation}
    and obviously $\sigma \leq 0$. Now, we show that $\sigma=0$. Otherwise, if $\sigma < 0$, then by continuity, there exists $R^* > 0$ such that
    \begin{equation}\label{eq:20251119-1952}
        u'(r) < \frac{\sigma}{2} r^{-1} \quad \text{for all } r \geq R^*.
    \end{equation}
    Let $h(r) := u(r) - \frac{\sigma}{2} \log r$ for $r \geq R^*$, then
    \[
    h'(r) = u'(r) - \frac{\sigma}{2} r^{-1} < 0,\,\,r \geq R^*,
    \]
    which implies that $h$ is strictly decreasing on $[R^*, \infty)$. Hence, $\lim_{r \rightarrow \infty} h(r) \leq h(R^*) < \infty$. However, since $u(r) \rightarrow 0$ and $\log r \rightarrow \infty$ as $r \rightarrow \infty$ (and $\sigma < 0$), we have
    \[
    \lim_{r \rightarrow \infty} h(r) = \lim_{r \rightarrow \infty} \left( u(r) - \frac{\sigma}{2} \log r \right) = +\infty,
    \]
   which is a contradiction. So $\sigma = 0$ and by \eqref{eq:20251119-1950} we get that $A(0) = (m - 1) \sigma = 0$.

    Since $A$ is non-increasing on $(0, b)$ and $A(0) = 0$, it follows that $A(b) \leq 0$.
\end{proof}

As in \cite{SeT}, for $a\in \R$, we introduce the function
$P$, defined for $0<u<\alpha$, by
\beq\lab{eq:20240919-1928}
P(u)=au\omega+r^N\left(\frac{m-1}{m}\frac{1}{|r'|^m}+\frac{a}{N}uf(u)
\right),
\eeq
where
\beq\lab{eq:20240919-1929}
\omega=\omega(u):=\frac{r^{N-1}}{r'|r'|^{m-2}}.
\eeq
Then we have the following property as \cite[Proposition 2.5]{SeT}.

\bo\lab{prop:20240919-2000} Let
$m>1$ and $1\leq N\leq m$. If $a\neq 0$,
\beq\lab{eq:20240919-1959}
P(u)=\left(a+1-\frac{N}{m}\right)\int_\alpha^u \omega(\tau)\mathrm{d}\tau+\frac{a}{N}\int_\alpha^u r(\tau)^N f(\tau)\left(g(\tau)-\frac{N-a}{a}\right)\mathrm{d}\tau,
\eeq
and if $a=0$,
\beq\lab{eq:20240921-1157}
P(u)=-\frac{N-m}{m}\int_\alpha^u \omega(\tau)\mathrm{d}\tau-\int_\alpha^u r(\tau)^N f(\tau)\mathrm{d}\tau,
\eeq
where $\alpha=u(0)$ and $0<u<\alpha$.
\eo
\bp
Recalling the definition of~$A(u)$ in~\eqref{A},
we rewrite $P$ as
\beq\lab{eq:20240919-1930}
P(u)=\left(a+1-\frac{N}{m}\right)u\omega(u) +\frac{1}{m}A(u)\omega(u) +\frac{a}{N}r^N u f(u).
\eeq
By \eqref{eq:20240919-1817} and  \eqref{eq:20240919-1929},
\beq\lab{eq:20240919-1935}
\begin{aligned}
w'(u)=&\frac{(N-1)r^{N-2}{|r'|}^m -r^{N-1}|r'|^{m-2}(m-1)r''}{{|r'|}^{2m-2}}\\
=&\frac{(N-1)r^{N-2}{|r'|}^m -r^{N-1}|r'|^{m-2}[\frac{N-1}{r}{r'}^2+f(u)|r'|^m r']}{{|r'|}^{2m-2}}\\
=&-f(u)r' r^{N-1}.
\end{aligned}
\eeq
Combing \eqref{eq:20240919-1922}, \eqref{eq:20240919-1930} with \eqref{eq:20240919-1935}, by a direct computation,
\begin{align*}
P'(u)=&\left(a+1-\frac{N}{m}\right)\omega+\left(a+1-\frac{N}{m}
\right)u\omega'
+\frac{1}{m}A'\omega+\frac{1}{m}A\omega'\\
&+\frac{a}{N}\left[Nr^{N-1}r' uf(u)+r^N f(u)+r^N uf'(u)\right]\\
=&\left(a+1-\frac{N}{m}\right)\omega(u) -r^Nf(u)+\frac{a}{N}r^Nf(u)+\frac{a}{N}r^N uf'(u).
\end{align*}
Noting that $u\rightarrow \alpha$ if and only if $r\rightarrow 0$. As $u\rightarrow \alpha$, i.e., $r\rightarrow 0$, by $0>u'(r)\rightarrow 0$, we have that $r'(u)\rightarrow -\infty$. Hence,
\beq\lab{eq:20240919-1952}
\omega(u)=\frac{r^{N-1}}{r'|r'|^{m-2}}\rightarrow 0,\quad \frac{r^N}{|r'|^m}\rightarrow 0~\hbox{ as }~r\rightarrow 0.
\eeq
Combining with \eqref{eq:20240919-1928}, we conclude that $P(\alpha)=0$.
{Thus}, if $a=0$, {then
$$P'(u)=\dfrac{m-N}{m}\omega(u)-r^Nf(u).$$ Hence,}
$P(u)=\int_\alpha^u P'(\tau)\mathrm{d}\tau$ implies \eqref{eq:20240921-1157}. {While,} if $a\neq 0$,
{since} $g(u):={uf'(u)/f(u)}$, {then}
\beq\lab{eq:20240919-1945}
P'(u)=\left(a+1-\frac{N}{m}\right)\omega(u) +\frac{a}{N}r^Nf(u)\left[g(u)-\frac{N-a}{a}\right].
\eeq
{In conclusion,} \eqref{eq:20240919-1959} is true, { as required.}
\ep

\subsection{{More preliminaries and a overview proof of Theorem \ref{th:20250107-1127} for $2\leq N\leq m$}}
In the following, let $u_1, u_2$ be two distinct radial solutions of Eq.~\eqref{eq:20240919-1424}. Put $\alpha_i=u_i(0)$.
Then, Proposition~\ref{prop:20240919-1425} implies
that $\alpha_i>b$, $i=1,2$. Let $r_i=r_i(u)$ be the inverse
of $u_i=u_i(r)$. Hence, $r_i=r_i(u)$ is defined
on $(0,\alpha_i)$, with $r_i(\alpha_i)=0$.

Without loss of generality, we suppose that $\alpha_1<\alpha_2$.
{For  $u\in (0,\alpha_1]$, we introduce more auxiliary functions $t,s,T$ and $S$, as in \cite{SeT}:}
\beq\lab{eq:20240920-0824}
\begin{aligned}
&t(u)=r_1(u)-r_2(u),\quad s(u)=\frac{r_1(u)}{r_2(u)},\\
&T(u)=\frac{r_1(u)}{r'_1(u)}-\frac{r_2(u)}{r'_2(u)}=\frac{1}{m-1}[A_1(u)-A_2(u)], \quad S(u)=\frac{\omega_1(u)}{\omega_2(u)},
\end{aligned}
\eeq
where $A_i$ is defined in~\eqref{A} for $r_i$ and $\omega_i$ is given in~\eqref{eq:20240919-1929} for $r_i$, $i=1,2$. {We also let $P_i$ be defined as in \eqref{eq:20240919-1928} after replacing $u$ by $u_i, i=1,2$ respectively.}
Here we see that
$$\dfrac{r_1(u)}{r'_1(u)}=0\mbox{ at }u=\alpha_1,\mbox{ since $r_1(\alpha_1)=0$ and }\lim\limits_{u\rightarrow \alpha_1}r'_1(u)=-\infty.$$

{The proof derives a contradiction by evaluating $P_1(u_c)-S(u_c)P_2(u_c)$ at a point $u_c\geq b$, defined in \eqref{eq:20240923-1132}, in two distinct ways. First, the original definition of $P$ in \eqref{eq:20240919-1928} shows via \eqref{eq:20240923-1402} that this expression is nonnegative. However, by employing an alternative, equivalent definition from Proposition~\ref{prop:20240919-2000}, we demonstrate that the same expression must be strictly negative. To establish negativity, we decompose the integral into a sum over carefully chosen intervals. Under assumption~\ref{hcon_3}, the sign of each constituent integral is analyzed, leading to the strict negative conclusion.}

Firstly, we establish the following fundamental properties for the auxiliary functions $t,s,T$ and $S$.
\begin{lemma}\label{lemma:20260123-2021}
Let $2\leq N\leq m$ and $t,s,T,S$ are defined by \eqref{eq:20240920-0824}. Then the following conclusions hold:
\begin{enumerate}[label=(\roman*)]
\item \label{prop:20240920-1125}
$t$ has no degenerate critical points in $(0,\alpha_1)$. Precisely, if $t'(u)=0$ {and} $t(u)>0$, then $t''(u)<0$. Similarly, if $t'(u)=0$ {and} $t(u)<0$, then $t''(u)>0$. Hence, $t$ has no positive local minimum (or negative local maximum) in $(0,\alpha_1)$.
\item \label{prop:20240920-1237} for all $u\in (0,\alpha_1)$
\beq\lab{eq:20240920-1240}
S'(u)=\left(\frac{r_1}{r_2}\right)^{N-1} \left(\frac{r'_2}{r'_1}\right)^{m-1}\!\! f(u) \big(|r'_2|^m-|r'_1|^m
\big).
\eeq
\item \label{cro:20240920-1249} The derivative
{$$S'(u)t'(u)\begin{cases}
\leq 0,\quad&\hbox{if}~u\in (0,b],\\
\geq 0,\quad&\hbox{if}~u\in (b,\alpha_1).
\end{cases}$$
In particular, $S'(u)=0$ if and only if $t'(u)=0$ on $(b,\alpha_1)$.}
\item \label{lemma:20240920-1517} The function
$$T(u)=\dfrac{1}{m-1}\big[A_1(u)-A_2(u)\big]
=-\dfrac{r_2^2}{r'_1r'_2}s'(u)$$ and $T'(u)=\dfrac{1}{m-1}f(u)\big(r_1|r'_1|^{m-1}-r_2|r'_2|^{m-1}
\big)$. In particular, $T_0:=\lim\limits_{u\downarrow 0}T(u)=0$.
\end{enumerate}
\end{lemma}
\bp
$\ref{prop:20240920-1125}$ If $u_c\in (0,\alpha_1)$ is a critical point of $t$, i.e., $r'_1(u_c)=r'_2(u_c)$, then by Proposition \ref{prop:20240919-2007}, we see that $r_1(u_c)\neq r_2(u_c)$, i.e., $t(u_c)\neq 0$. A direct computation yields that
\begin{align*}
t''(u_c)=&\frac{1}{m-1}\left[\frac{N-1}{r_1}{r'_1}^2+f(u)|r'_1|^m r'_1\right]-\frac{1}{m-1}\left[\frac{N-1}{r_2}{r'_2}^2+f(u)|r'_2|^m r'_2\right]\\
=&\frac{N-1}{m-1}{r'_1}^2\left(\frac{1}{r_1}-\frac{1}{r_2}\right).
\end{align*}
Hence, if $t(u_c)>0$, i.e., $r_1(u_c)>r_2(u_c)>0$, then $t''(u_c)<0$. {This} implies that $u_c$ is a local strict maximum of $t$.
Similarly, if $t(u_c)<0$, i.e., $0<r_1(u_c)<r_2(u_c)$, then $t''(u_c)>0$, which implies that $u_c$ is a local strict minimum of $t$.

$\ref{prop:20240920-1237}$ {Recalling that $\omega'_i(u)=-r_{i}^{N-1}r'_if(u)$, $i=1,2$,
by~\eqref{eq:20240919-1935}, by a direct computation we get that}
\begin{align*}
S'(u)=&{\left(\frac{\omega_1(u)}{\omega_2(u)}\right)'=\dfrac{\omega'_1(u)\omega_2(u)-\omega_1(u)\omega'_2(u)}
{\omega_2(u)^2}}\\
=&\dfrac{-r_{1}^{N-1}r'_1 f(u)\dfrac{r_{2}^{N-1}}{r'_2|r'_2|^{m-2}} + r_{2}^{N-1}r'_2f(u)\dfrac{r_{1}^{N-1}}{r'_1|r'_1|^{m-2}}}
{\left(\dfrac{r_{2}^{N-1}}{r'_2|r'_2|^{m-2}}\right)^2}\\
=&\frac{r_{1}^{N-1}r_{2}^{N-1}f(u)\left(
\dfrac{r'_2}{r'_1|r'_1|^{m-2}}-\dfrac{r'_1}{r'_2|r'_2|^{m-2}}
\right)}{\left(\dfrac{r_{2}^{N-1}}{r'_2|r'_2|^{m-2}}\right)^2}\\
=&\left(\dfrac{r_1}{r_2}\right)^{N-1} \left(\dfrac{r'_2}{r'_1}\right)^{m-1} f(u) \big(|r'_2|^m-|r'_1|^m
\big),
\end{align*}
as stated.

$\ref{cro:20240920-1249}$ Since $t'(u)=r'_1-r'_2=|r'_2|-|r'_1|$, one can see that
\beq\lab{eq:20250108-1819}
\begin{gathered}
t'(u)>(\mbox{resp.} =,\leq )\,0\,\Leftrightarrow\, |r'_2|^m-|r'_1|^m>(\mbox{resp.} =,\leq )\,0\\
\Leftrightarrow r'_2<(\mbox{resp.} =,\geq )\,r'_1<0.
\end{gathered}
\eeq
{Hence, $f(u)\leq 0$ for $u\in (0,b]$, which, combined with $r_1(u)>0$, $r_2(u)>0$, $r'_1(u)<0$, $r'_2(u)<0$, implies that $\left(\dfrac{r_1}{r_2}\right)^{N-1} \left(\dfrac{r'_2}{r'_1}\right)^{m-1} f(u)\leq 0$. Thus, by the conclusion in $\ref{prop:20240920-1237}$, we see that $S'(u)t'(u)\leq 0$ for $u\in (0,b]$.}

{Similarly, by $f(u)>0$ for $u\in (b, \alpha_1)$, we can deduce that $\left(\dfrac{r_1}{r_2}\right)^{N-1} \left(\dfrac{r'_2}{r'_1}\right)^{m-1} f(u)>0$ in $(b,\alpha_1)$.
Then by the conclusion in $\ref{prop:20240920-1237}$ again, we obtain that $S'(u)t'(u)\geq 0$ for $u\in (b,\alpha_1)$. In particular, $S'(u)=0$ if and only if $r'_1(u)=r'_2(u)$, i.e., $t'(u)=0$.}

$\ref{lemma:20240920-1517}$ Noting that $s'(u)=\left(\dfrac{r_1}{r_2}\right)'=\dfrac{r'_1r_2-r_1r'_2}{r_2^2}$, we have
$$T(u)=\dfrac{r_1}{r'_1}-\dfrac{r_2}{r'_2}=-\dfrac{r'_1r_2-r_1r'_2}{r'_1r'_2}=-\frac{r_2^2}{r'_1r'_2} s'(u).$$
Furthermore, by \eqref{eq:20240919-1817}, we have
\beq\lab{eq:20240920-1744}
\begin{aligned}
T'(u)=&\left(\dfrac{r_1}{r'_1}\right)'- \left(\dfrac{r_2}{r'_2}
\right)'
=\frac{{r'_1}^2-r_1r''_1}{{r'_1}^2}-\frac{{r'_2}^2-r_2r''_2}{{r'_2}^2}
=\frac{r_2r''_2}{{r'_2}^2}-\frac{r_1r''_1}{{r'_1}^2}\\
=&\frac{1}{m-1}\frac{r_2}{{r'_2}^2}\left(\frac{N-1}{r_2}{r'_2}^2
+f(u)|r'_2|^m r'_2\right)\\
&\qquad-\frac{1}{m-1}\frac{r_1}{{r'_1}^2}\left(\frac{N-1}{r_1}{r'_1}^2
+f(u)|r'_1|^m r'_1\right)\\
=&\frac{1}{m-1}f(u)\big(r_2|r'_2|^{m-2}r'_2-r_1|r'_1|^{m-2}r'_1
\big)\\
=&{\frac{1}{m-1}}f(u)\big(r_1|r'_1|^{m-1}-r_2|r'_2|^{m-1}\big).
\end{aligned}
\eeq
Lemma \ref{lemma:20240919-1907} yields that $A_1(u)\rightarrow 0$, $A_2(u)\rightarrow 0$ as $u\downarrow 0$. Hence, $$T_0:=\lim\limits_{u\downarrow 0}T(u)=\lim\limits_{u\downarrow 0}\frac{1}{m-1}\big[A_1(u)-A_2(u)\big]=0,$$
{as required.}
\ep

\begin{lemma}\label{lemma:20260123-2113}
Let $2\leq N\leq m$, then
\begin{enumerate}[label=(\roman*)]
\item \label{prop:20240920-1549}
$t(u)t'(u)<0$ on $(0,b]$.
\item \label{prop:20240923-1336}
$t$ has at least one zero in $(b,\alpha_1)$.
\end{enumerate}
\end{lemma}

\bp
\ref{prop:20240920-1549}
First, we {\bf claim} that $t=t(u)$ has at most one zero in $(0,b]$. If not, since $t(u)\not\equiv 0$, we can find two adjacent zeros $\tau_1, \tau_2\in (0,b], \tau_1<\tau_2$ and $t(u)>0$ for $u\in (\tau_1, \tau_2)$ (see
Remark~\ref{remark:20241010-1521}).
{Hence,} there exists some $\tau\in (\tau_1,\tau_2)$ such that $t'(\tau)=0$. {Therefore,} by
Lemma \ref{lemma:20260123-2021}-\ref{prop:20240920-1125}, there is no other critical point of $t$ in $(\tau, \tau_2)$. {Thus, $t(\tau)>0$ and $t(\tau_2)=0$
imply} that $t'(u)<0$ for $u\in (\tau, \tau_2)$
{and so} $t'(\tau_2)\leq 0$. Furthermore,
Proposition~\ref{prop:20240919-2007} and $t(\tau_2)=0$
{yield} that $t'(\tau_2)<0$.

For $u\in (\tau, \tau_2]\subset (0,b]$, it holds
that $t'(u)< 0$. Then by
Lemma \ref{lemma:20260123-2021}-\ref{cro:20240920-1249}, we
{derive that} $S'(u)\geq 0$ for $u\in (\tau, \tau_2]$. {This} implies that $S(\tau_2)\geq S(\tau)$.

{However, at $u=\tau$, by $t'(\tau)=0$ and $t(\tau)>0$, we have that $r'_1(\tau)=r'_2(\tau)$, $r_1(\tau)>r_2(\tau)>0$ and $s(\tau)^{N-1}>1$.}
Thus,
$$S(\tau)=\frac{{r_{1}^{N-1}}/{r'_1|r'_1|^{m-2}}}
{{r_{2}^{N-1}}/{r'_2|r'_2|^{m-2}}}=\left(\dfrac{r_1(\tau)}{r_2(\tau)}
\right)^{N-1}=s(\tau)^{N-1}>1.$$
{And at $u=\tau_2$, it holds $t(\tau_2)=0$, which implies that $r_1(\tau_2)=r_2(\tau_2)$.} Since $t'(\tau_2)<0$, we have $r'_1(\tau_2)<r'_2(\tau_2)<0$ and so $\dfrac{|r'_2(\tau_2)|}{|r'_1(\tau_2)|}<1$. Hence,
$$S(\tau_2)=\frac{{r_{1}^{N-1}}/{r'_1|r'_1|^{m-2}}}{{r_{2}^{N-1}}/{r'_2|r'_2|^{m-2}}}=\frac{r'_2(\tau_2)|r'_2(\tau_2)|^{m-2}}{r'_1(\tau_2)|r'_1(\tau_2)|^{m-2}}=\left(\frac{|r'_2(\tau_2)|}{|r'_1(\tau_2)|}\right)^{m-1}<1.$$
This contradicts  $S(\tau_2)\geq S(\tau)$ and proves the claim.

Now, without loss of generality, we assume that there exists some $\delta\in (0,b]$ such that
\beq\lab{eq:20240920-1619}
t(u)=r_1(u)-r_2(u)>0~\mbox{ for }~0<u\leq \delta.
\eeq
Let $B(u)$ be defined by \eqref{eq:20240919-2033}.
{A direct computation and \eqref{eq:20240919-1817} yield}
\begin{align*}
B'(u)=&-m\dfrac{r''_1}{|r'_1|^m r'_1}+m\dfrac{r''_2}{|r'_2|^m r'_2}\\
=&\frac{m}{m-1}\frac{1}{|r'_2|^m r'_2}\left(\frac{N-1}{r_2}{r'_2}^2+f(u)|r'_2|^m r'_2\right)
-\frac{m}{m-1}\frac{1}{|r'_1|^m r'_1}\left(\frac{N-1}{r_1}{r'_1}^2+f(u)|r'_1|^m r'_1
\right)\\
=&\frac{m}{m-1}(N-1)\left(\frac{1}{|r'_2|^{m-2}r'_2 r_2}-\frac{1}{|r'_1|^{m-2}r'_1 r_1}\right)\\
=&\frac{m}{m-1}(N-1)\left(\frac{1}{|r'_1|^{m-1} r_1}-\frac{1}{|r'_2|^{m-1} r_2}\right),
\end{align*}
which implies that
\beq\lab{eq:20250108-1923}
\frac{r_2}{|r'_2|}B'(u)=\frac{m}{m-1}(N-1)
\left(\frac{r_2}{r_1}\frac{1}{|r'_1|^{m-1} |r'_2|}-\frac{1}{|r'_2|^{m}}\right).
\eeq
Noting that \eqref{eq:20240920-1619} implies  $\dfrac{r_2}{r_1}<1$ on $(0,\delta]$, thus
\beq\lab{eq:20250108-1924}
\frac{r_2}{|r'_2|}B'(u)\leq \frac{m}{m-1}(N-1)\left(\frac{1}{|r'_1|^{m-1} |r'_2|}-\frac{1}{|r'_2|^{m}}\right).
\eeq
 By Young's inequality,  ${\dfrac{1}{|r'_1|^{m-1} |r'_2|}\leq} \dfrac{m-1}{m}\dfrac{1}{|r'_1|^m}
 +\dfrac{1}{m}\dfrac{1}{|r'_2|^m} $,
 {so} that, for $u\in (0,\delta]$,
\beq\lab{eq:20240920-1648}\begin{aligned}
\frac{m}{m-1}(N-1)\left(\frac{1}{|r'_1|^{m-1} |r'_2|}-\frac{1}{|r'_2|^{m}}\right)
&\leq(N-1) \left(\frac{1}{|r'_1|^m}-\frac{1}{|r'_2|^m}\right)\\
&=(N-1)B(u).\end{aligned}
\eeq
Hence,
\beq\lab{eq:20240920-1649}
B'(u)<(N-1)\frac{|r'_2|}{r_2} B(u), \quad u\in (0,\delta].
\eeq
Lemma \ref{lemma:20260123-2021}-\ref{prop:20240920-1125} gives that $t(u)$ has no local minimum on $(0,\delta]$. That is, $t(u)$ can have at most one critical point necessarily a strict maximum in $(0,\delta)$. Thus, {either}
\beq\lab{eq:20240920-1630}
t'(u)=r'_1(u)-r'_2(u)<0~\mbox{ for }~0<u\leq \delta
\eeq
or there exists some $\delta_1\in (0,\delta]$ such that
\beq\lab{eq:20240920-1632}
t'(u)=r'_1(u)-r'_2(u)>0~\mbox{ for }~0<u<\delta_1.
\eeq
We {\bf claim} that \eqref{eq:20240920-1632} is impossible. If not, $0>r'_1(u)>r'_2(u)$ and {so} $B(u)>0$
for $0<u<\delta_1$.
{This combined with
\eqref{eq:20240920-1649} implies that}
\beq\lab{eq:20240920-1656}
\frac{B'(u)}{B(u)}<(N-1)\frac{|r'_2|}{r_2},\quad u\in (0, \delta_1).
\eeq
For any {$u$, $\bar{u}$, with} $0<u<\bar{u}<\delta_1$,   integrating \eqref{eq:20240920-1656} from $u$ to $\bar{u}$, we {have}
\beq\lab{eq:20240920-1702}
\frac{B(\bar{u})}{B(u)}<\left(\frac{r_2(u)}{r_2(\bar{u})}\right)^{N-1}~\hbox{for any}~0<u<\bar{u}<\delta_1.
\eeq
{Thus, $B(u)>0$ and \eqref{eq:20240920-1702} yield} that
\beq\lab{eq:20240920-1703}
0<r_{2}^{N-1}(\bar{u})B(\bar{u})<r_{2}^{N-1}(u)B(u), \forall~ 0<u<\bar{u}<\delta_1.
\eeq
{Proposition \ref{prop:20240919-1819}-(i) gives in particular that} $\dfrac{r_{i}^{N-1}(u)}{|r'_i(u)|^{m-1}}\rightarrow 0$ and $\dfrac{1}{|r'_i(u)|}\rightarrow 0$ as $u\downarrow 0$ for $i=1,2$. Then, assumption \eqref{eq:20240920-1619} {implies}
\beq\lab{eq:20240920-1718}
\frac{r_{2}^{N-1}(u)}{|r'_1(u)|^m}
<\frac{r_{1}^{N-1}(u)}{|r'_1(u)|^m}\rightarrow 0~\hbox{as}~u\downarrow 0.
\eeq
Hence,
\beq\lab{eq:20240920-1719}
\lim_{u\downarrow 0} r_{2}^{N-1}(u)B(u)
=\lim_{u\downarrow 0} \left[\frac{r_{2}^{N-1}(u)}{|r'_1(u)|^m}
-\frac{r_{2}^{N-1}(u)}{|r'_2(u)|^m}\right]=0.
\eeq
{Thus}, letting $u\downarrow 0$ in \eqref{eq:20240920-1703}, we obtain that $$0<r_{2}^{N-1}(\bar{u})B(\bar{u})\leq \lim\limits_{u\downarrow 0} r_{2}^{N-1}(u)B(u)=0,$$
{which is an obvious} contradiction. The claim is
{so} proved and \eqref{eq:20240920-1630} is true.

{Now,} \eqref{eq:20240920-1619}   and
\eqref{eq:20240920-1630}
{imply that $t>0$ and} $t'<0$ for $u\in (0,\delta]$.  If $tt'<0$ were not true on the whole interval $(0,b]$, we define
\beq
\delta_2:=\sup\{{v>0}\,:\, t(u)t'(u)<0~\mbox{ for all }~0<u<v\}.
\eeq
We note that $\delta<\delta_2\leq b$ and $t(\delta_2)t'(\delta_2)=0$. Then it is easy to see that one of the following {cases} holds:
\begin{itemize}
\item[(i)] $t(\delta_2)>0$, $t'(\delta_2)=0$, $t''(\delta_2)\geq 0$;
\item[(ii)] $t(\delta_2)=0$, $t'(\delta_2)=0$;
\item[(iii)] $t(\delta_2)=0$, $t'(\delta_2)<0$.
\end{itemize}
Case (i) cannot occur, since it contradicts Lemma \ref{lemma:20260123-2021}-\ref{prop:20240920-1125}. Also (ii) cannot happen, since it goes against Proposition~\ref{prop:20240919-2007}. Thus, (iii) holds.

{The definition of $\delta_2$ implies that}
\beq\lab{eq:20240920-1742}
\begin{cases}
t(u)>0,\mbox{ i.e., }~r_1>r_2,~&\mbox{for }~0<u<\delta_2\\
t'(u)<0, \mbox{ i.e., }~|r'_1|>|r'_2|, ~&\mbox{for }~0<u\leq \delta_2.
\end{cases}
\eeq
In particular, $r'_1(\delta_2)<r'_2(\delta_2)<0$ and {so}
\beq\lab{eq:20240920-1735}
T(\delta_2)=\frac{r_1(\delta_2)}{r'_1(\delta_2)}
-\frac{r_2(\delta_2)}{r'_2(\delta_2)}
=r_1(\delta_2)\left(\frac{1}{|r'_2(\delta_2)|}
-\frac{1}{|r'_1(\delta_2)|}\right)>0.
\eeq
{For} $0<u<\delta_2\leq b$, it holds that $f(u)\leq 0$, so \eqref{eq:20240920-1744} and \eqref{eq:20240920-1742} give
\beq\lab{eq:20240920-1750}
T'(u)=f(u)\big(r_1|r'_1|^{m-1}-r_2|r'_2|^{m-1}\big)\leq 0.
\eeq
Hence,
$$T_0:=\lim_{u\downarrow 0}T(u)\geq T(\delta_2)>0.$$
This contradicts the fact that $T_0=0$ by
Lemma \ref{lemma:20260123-2021}-\ref{lemma:20240920-1517} and completes the proof.

\ref{prop:20240923-1336} Without loss of generality, we assume that $b<\alpha_1<\alpha_2$ {as before, then} we note that
\beq\lab{eq:20240921-0834}
t(\alpha_1)<0,\quad t'(\alpha_1)=-\infty.
\eeq
{By \ref{prop:20240920-1549}, $t$ has no zero in $(0,b]$. Being  $t(\alpha_1)<0$, arguing by contradiction, we suppose that}
\beq\lab{eq:20240923-1046}
t(u)=r_1(u)-r_2(u)<0~\hbox{over the entire interval $(0,\alpha_1)$}.
\eeq
Lemma \ref{lemma:20260123-2113}-\ref{prop:20240920-1549} gives that $t'(u)>0$ over $(0,b]$. Combining {this} with $t'(\alpha_1)=-\infty$, we conclude the existence of
{a} critical point of $t(u)$. {Thus,} there exists  some $\bar{u}\in {(b,\alpha_1)}$ such that
\beq\lab{eq:20240921-1549}
t(\bar{u})<0,\;\;t'(\bar{u})=0~\mbox{ and }~t'(u)<0~, \mbox{ for all }u\in (\bar{u},\alpha_1)
\eeq
This implies that $t''(\bar{u})\leq 0$ and contradicts Lemma \ref{lemma:20260123-2021}-\ref{prop:20240920-1125}.
\ep

\br\lab{remark:20241010-1521}
{\rm Let us note that the proof of
Lemma \ref{lemma:20260123-2113}-\ref{prop:20240920-1549} does not depend on the assumption $\alpha_1<\alpha_2$. Thus, without loss of generality, we may assume that $t(u)>0$ for $u\in (\tau_1, \tau_2)$. Hence,  \eqref{eq:20240920-1619} {holds.
Otherwise}, only limited to the proof process of
Lemma \ref{lemma:20260123-2113}-\ref{prop:20240920-1549}, we replace $t(u)$ by $\tilde{t}(u):=r_2(u)-r_1(u)$ and then the arguments above lead to $\tilde{t}(u)\tilde{t}'(u)<0$ on $(0,b]$. Since $\tilde{t}(u)=-t(u)$, we also obtain that $t(u)t'(u)<0$ on $(0,b]$.}
\er

\subsection{Proof of Theorem $\ref{th:20250107-1127}$ for $2\leq N\leq m$}

{Combining Lemma \ref{lemma:20260123-2113}-$\ref{prop:20240923-1336}$
with $t(\alpha_1)<0$,} by~\eqref{eq:20240921-0834},
we define
\beq\lab{eq:20240923-1120}
u_I:=\inf \{u\,:\, t(v)<0~\mbox{ for all }~u<v<\alpha_1\}.
\eeq
{Hence,}
\beq\lab{eq:20240923-1133}
t(u_I)=0~\hbox{and}~t(u)<0~\hbox{for}~u_I<u<\alpha_1.
\eeq
That is, $u_I$ denotes the largest zero of $t$ in $(0,\alpha_1)$.
Recalling \eqref{eq:20240921-0834} again that $t'(\alpha_1)=-\infty$, we {also define}
\beq\lab{eq:20240923-1132}
\begin{aligned}
u_c:=\sup\{u\geq b\,:\, \mbox{there exists some } \varepsilon_u>0&\mbox{ such that }t(u)\geq t(v)\mbox{ for all}\\
&v\in [b,\alpha_1)\cap (u-\varepsilon_u,u+\varepsilon_u)\}.
\end{aligned}
\eeq
That is, if there exists some local maximum point
{for $t$} in $[b,\alpha_1)$, then $u_c$ denotes the largest one. Otherwise, $u_c=b$.

{
\textbf{Step 1.} We first prove the inequality
\beq\lab{eq:20240923-1402}
P_1(u_c) - S(u_c)P_2(u_c)
\begin{cases}
\geq 0, & \text{if } u_c = b, \\[2pt]
> 0,   & \text{if } u_c > b,
\end{cases}
\eeq
where $P_i$ ($i=1,2$) are defined as in \eqref{eq:20240919-1928} with $u$ replaced by $u_i$, respectively. This conclusion will be supported by the following lemmas.}

\bl\lab{lemma:20240923-1227}
$b\leq u_c<u_I$.
\el

\bp
It is trivial that $b\leq u_c$. {We only need to show} that $u_c<u_I$. {Arguing by contradiction, we} suppose that $u_c\geq u_I$.

{\bf Case 1:} $u_c=u_I$. Then $t(u_c)=0$, which implies
that $u_c\neq b$, {being} $t(b)t'(b)<0$. {Hence,} $u_c>b$
and {so} $t'(u_c)=t(u_c)=0$, {which is impossible by}
Proposition~\ref{prop:20240919-2007}.

{\bf Case 2:} $u_c>u_I$. By the definition of $u_I$, it holds that $t(u_c)<0$. {Therefore,} $t(u_c)$ is a negative local maximum. This is again impossible now by
Lemma \ref{lemma:20260123-2021}-\ref{prop:20240920-1125}.
\ep

\begin{lemma}\lab{prop:20240923-1148}
It holds that
\beq\lab{eq:20240923-1246}
t(u_c)>0 ~\mbox{ and }~t'(u)<0 \mbox{ for all }u_c<u<\alpha_1.
\eeq
\end{lemma}

\bp
We first prove that the property holds for $u>u_c$ close to $u_c$.\\
{\bf Case $u_c=b$:} In such a case, by the definition of $u_c$, one can see that  $t(u)<t(b)$ for all $u\in (b,\alpha_1)$. This implies that $t'(b)\leq 0$.
However,  $t(u)t'(u)<0$ on $(0,b]$ by
Lemma \ref{lemma:20260123-2113}-\ref{prop:20240920-1549}, so that $t'(b)\neq 0$. Thus, $t'(b)<0$ {and so} $t(b)>0$. In particular, $t'(u)<0$ for $u>u_c$ close to  $u_c$.\\
{\bf Case $u_c>b$:} In such a case, it holds that $t'(u_c)=0$ and $t''(u_c)\leq 0$. In particular, from Lemma \ref{lemma:20260123-2021}-\ref{prop:20240920-1125} we deduce further that $t''(u_c)< 0$ and $t(u_c)>0$. Hence, {again} $t'(u)<0$ for $u>u_c$ close to  $u_c$.
\vskip 0.02in

If $t'(u)<0$ {were} not true for all $u\in (u_c,\alpha_1)$, then there exists a first $u^*_1>u_c$ such that $t'(u^*_1)=0$ with $t''(u^*_1)\geq 0$. Furthermore,  $t''(u^*_1)>0$ and $t(u^*_1)<0$ by Lemma \ref{lemma:20260123-2021}-\ref{prop:20240920-1125}. By \eqref{eq:20240921-0834} there exists some local maximum point $u^*_2\in (u_1^*,\alpha_1)$. Then, the
definition~\eqref{eq:20240923-1132} of $u_c$ yields
$$u_c\geq u^*_2>u_1^*,$$
which {is impossible, being} $u_1^*>u_c$.
\ep

\begin{lemma}\lab{prop:20240923-1323}
Any critical point of $s(u)$ in $(u_c, \alpha_1)$ must be a strict local maximum point. That is, if $s'(u)=0$ for some $u\in (u_c, \alpha_1)$, then $s''(u)<0$.
\end{lemma}
\bp
Since $s(u)=\dfrac{r_1(u)}{r_2(u)}$, we {obtain} $s'(u)=\dfrac{r'_1r_2-r_1r'_2}{r_2^2}$ and $$
s''(u)=\dfrac{\big(r''_1r_2-r_1r''_2\big)r_2^2-2r_2r'_2
\big(r'_1r_2-r_1r'_2\big)}{r_2^4}.
$$
{Hence}, if $s'(u)=0$, i.e., $r'_1r_2-r_1r'_2=0$, we deduce that
\beq\lab{eq:20240921-0942}
\begin{aligned}
s''(u)=&\frac{r''_1r_2-r_1r''_2}{r_2^2}\\
=&\frac{1}{m-1}\frac{1}{r_2^2}\left\{r_2\left[\frac{N-1}{r_1}{r'_1}^2
+f(u)|r'_1|^m r'_1\right]\right.\\
&\qquad\qquad\qquad\left.-r_1\left[\frac{N-1}{r_2}{r'_2}^2
+f(u)|r'_2|^m r'_2\right]\right\}\\
=&\frac{1}{m-1}\frac{r'_1}{r_2}f(u)\big(|r'_1|^m-|r'_2|^m\big).
\end{aligned}
\eeq
{Now, $t'(u)<0$ for $u_c<u<\alpha_1$ by
Proposition~\ref{prop:20240923-1148},} that is, $r'_1(u)<r'_2(u)<0$ for $u_c<u<\alpha_1$. Hence, $|r'_1|^m-|r'_2|^m>0$ for $u_c<u<\alpha_1$. Recalling that $u_c\geq b$, under the assumption \ref{hcon_1}, {we see} that $f(u)>0$ for $u_c<u<\alpha_1$.  Combining {this with the facts
that $r'_1<0$ and $r_2>0$ in} $(0,\alpha_1)$, we conclude that $s''(u)<0$ for any $u\in(u_c,\alpha_1)$, with $s'(u)=0$,
{as required.}
\ep

%\begin{lemma}\lab{prop:20240923-1359}
%Let $P_i$ be defined {as in} \eqref{eq:20240919-1928} for $i=1,2$ respectively. Then
%\beq\lab{eq:20240923-1402}
%P_1(u_c)-S(u_c)P_2(u_c)\begin{cases}\geq 0,\quad &\hbox{if}~u_c=b,\\
%>0,&\hbox{if}~u_c>b.
%\end{cases}
%\eeq
%\end{lemma}

\textbf{Proof of  \eqref{eq:20240923-1402}}: We divide the proof into two cases:

{\bf Case I: $u_c=b$.}  In such a case, by
Lemma \ref{prop:20240923-1148}
and Lemma \ref{lemma:20260123-2113}-\ref{prop:20240920-1549}, we see that $t(b)>0$, $t'(b)<0$.  Furthermore, $t(u)>0$ and $t'(u)<0$ for all $0<u\leq b$, i.e., $r_1(u)>r_2(u)>0$ and $r'_1(u)<r'_2(u)<0$ for $u\in (0,b]$. Hence, $|r'_1(u)|>|r'_2(u)|$ in $(0,b]$.
{Thus}, $r_1|r'_1|^{m-1}-r_2|r'_2|^{m-1}>0$  {in} $(0,b]$. Lemma \ref{lemma:20260123-2021}-\ref{lemma:20240920-1517} says that $T(0)=T_0=0$ and  that $T'(u)=\dfrac{1}{m-1}f(u)\big(r_1|r'_1|^{m-1}-r_2|r'_2|^{m-1}
\big)$. {Since} $f(u)\leq 0$ for $u\in (0,b]$, we {deduce} that
\beq\lab{eq:20240921-0847}
T'(u)\leq 0~\mbox{ for }~0<u\leq b.
\eeq
In particular,
\beq\lab{eq:20240921-0848}
T(b)\leq T_0=0.
\eeq
Then, using the identity \eqref{eq:20240919-1930} for both the solutions $u_1$ and $u_2$, together with the
definition~\eqref{eq:20240920-0824} of $S(u)$ and
Lemma \ref{lemma:20260123-2021}-\ref{lemma:20240920-1517}, we get
\beq\lab{eq:20240921-0853}\begin{aligned}
P_1(b)-S(b)P_2(b)&=\frac{1}{m}[\omega_1(b)A_1(b)-S(b)
\omega_2(b)A_2(b)]\\
&=\frac{m-1}{m}\omega_1(b)T(b)\geq 0,\end{aligned}
\eeq
where we have used the fact {that} $\omega_1(b)=\dfrac{r_1(b)}{r'_1(b)}<0$.

{\bf Case II: $u_c>b$.} In such a case, it holds that $t'(u_c)=0$. Moreover,
Lemma \ref{prop:20240923-1148} and Lemma \ref{lemma:20260123-2021}-\ref{prop:20240920-1125}
give that $t(u_c)>0$ and $t''(u_c)<0$. {Hence}, $r_1(u_c)>r_2(u_c)>0$ and $r'_1(u_c)=r'_2(u_c)$
{This} implies that $s'(u_c)=\dfrac{r'_1(u_c)r_2(u_c)-r'_2(u_c)r_1(u_c)}{r_2^2(u_c)}>0$.
{Since}
$$S(u_c)=\frac{\omega_1(u_c)}{\omega_2(u_c)}
=\left(\frac{r_1(u_c)}{r_2(u_c)}\right)^{N-1}=s^{N-1}(u_c)>1,$$
we have $S(u_c)=s^{N-1}(u_c)<s(u_c)^N$, i.e., $r_1(u_c)^N-S(u_c)r_2(u_c)^N>0$.
{Thus,  $f(u_c)>0$ and \eqref{eq:20240919-1928}
imply that} for any $a>0$
\beq\lab{eq:20240921-1433}
\begin{aligned}
P_1(u_c)&-S(u_c)P_2(u_c)\\
&=au_c\omega_1(u_c)+r_1(u_c)^N
\left[\frac{m-1}{m}\frac{1}{|r'_1(u_c)|^m}
+\frac{a}{N}u_cf(u_c)\right]\\
&\quad-S(u_c)\left\{au_c\omega_2(u_c)
+r_2(u_c)^N\left[\frac{m-1}{m}
\frac{1}{|r'_2(u_c)|^m}+\frac{a}{N}u_cf(u_c)\right]\right\}\\
&=\left(\frac{m-1}{m}\frac{1}{|r'_2(u_c)|^m}+\frac{a}{N}u_cf(u_c)
\right)\cdot\big(r_1(u_c)^N-S(u_c)r_2(u_c)^N\big)>0.
\end{aligned}
\eeq
{This completes the proof of  \eqref{eq:20240923-1402}.}\hfill$\Box$

\vskip 0.2in

{
\textbf{Step 2.} We now establish the inequality $P_1(u_c) - S(u_c)P_2(u_c) < 0$ by an alternative method, thereby contradicting the nonnegativity result obtained in \textbf{Step 1}.
Using the equivalent integral representation of $P$ from Proposition~\ref{prop:20240919-2000}, we decompose the integral over carefully chosen partition points into a sum of contributions from distinct intervals.
To determine the sign of the integrand in each subinterval, we first develop several auxiliary lemmas.}

\begin{lemma}\lab{prop:20240923-1526} The derivative
$$s'(u_c)\begin{cases} \geq 0,\quad&\hbox{if}~u_c=b,\\
>0,\quad&\hbox{if}~u_c>b.
\end{cases}$$
\end{lemma}

\bp
If $u_c=b$, {then $T(b)\leq 0$ by} \eqref{eq:20240921-0848}.  {Moreover, $T(b)=-\dfrac{r_2(b)^2}{r'_1(b)r'_2(b)}s'(b)$ by Lemma~\ref{lemma:20240920-1517}, so} that $s'(u_c)=s'(b)\geq 0$.

If $u_c>b$, {then} $t(u_c)>0$, $t'(u_c)=0$ and $t''(u_c)<0$. {Hence}, $r_1(u_c)>r_2(u_c)>0$ and $r'_1(u_c)=r'_2(u_c)<0$. {This} implies that
$s'(u_c)=\dfrac{r'_1(u_c)r_2(u_c)
-r'_2(u_c)r_1(u_c)}{r_2^2(u_c)}>0$.
\ep

\begin{lemma}\lab{prop:20240923-1650}
There exists a {number} $\bar{u}_I\in [u_c, u_I)\subset [u_c, \alpha_1)$ such that $s'(\bar{u}_I)=0$ and
\beq\lab{eq:20240921-1012}
{s'(u)\begin{cases}
>0,\quad &u\in (u_c, \bar{u}_I);\\
<0,\quad& u\in (\bar{u}_I,\alpha_1).
\end{cases}}
\eeq
\end{lemma}

\bp
{Let us first} note that $s(u_c)=\dfrac{r_1(u_c)}{r_2(u_c)}>1$, since $t(u_c)>0$ by Proposition~\ref{prop:20240923-1148}. {Moreover,} $s(u_I)=1$, {since} $t(u_I)=0$. {Let $\bar{u}_I\in [u_c, u_I]$ be the number} such that
\beq\lab{eq:20240923-1502}
s(\bar{u}_I)=\max_{u\in [u_c, u_I]}s(u).
\eeq
We {assert}  that $\bar{u}_I$ is exactly the unique element we are seeking for.
{Proposition~\ref{prop:20240923-1323} shows that $s(u)$ cannot be} constant in any interval contained in $(u_c, \alpha_1)$.
Hence, we divide the proof into two different cases.

{\bf Case I: $s(u)<s(u_c)$ for $u>u_c$ close to $u_c$.} In such a case, we have that $s'(u_c)\leq 0$. {On the other hand},  Proposition~\ref{prop:20240923-1526} {shows that this occurs only when $u_c=b$ and} $s'(b)=0$. {Consequently,} $s'(u)<0$ for $u>b$ close to $b$. We shall prove that $s'(u)<0$ for all $u\in (b, \alpha_1)$ and then {that} $\bar{u}_I=u_c=b$. If not, let $u_0>b$ be the first zero of {$s'$}. Then, it holds that $s''({u_0})\geq 0$, which is impossible by Proposition~\ref{prop:20240923-1323}.

{\bf Case II: $s(u)>s(u_c)$ for $u>u_c$ close to $u_c$.} In such a case, recalling again $s(u_c)>1=s(u_I)$, we see that $\max_{u\in [u_c, u_I]}s(u)$ is attained by some $\bar{u}_I$. By Proposition~\ref{prop:20240923-1323} again, {$s$} has a unique critical point, which is also a local maximum, in $(u_c,\alpha_1)$. In particular, \eqref{eq:20240921-1012} holds.
\ep

\br\lab{remark:20240923-1555}
{\rm By the proof of Proposition \ref{eq:20240921-1012}, we see that either
\begin{gather}\lab{eq:20240923-1558}
b\leq u_c<\bar{u}_I<u_I<\alpha_1,\quad\mbox{or}\\
\lab{eq:20240923-1559}
b=u_c=\bar{u}_I<u_I<\alpha_1.
\end{gather}}
\er

\begin{lemma}\lab{prop:20240923-1600}
There exists some $u_e\in (\bar{u}_I, \alpha_1)$ such that $s(u_e)^N=S(u_c)$ and
\beq\lab{eq:20240923-1601}
{s(u)^N-S(u_c)\begin{cases}
> 0~\quad&\hbox{if}~u\in (u_c, u_e),\\
<0~&\hbox{if}~u\in (u_e, \alpha_1).
\end{cases}}
\eeq
\end{lemma}

\bp
If \eqref{eq:20240923-1559} {occurs}, then  $b=u_c=\bar{u}_I$ and $s'(b)=0$. {Hence},  $\dfrac{r'_1(b)}{r'_2(b)}=\dfrac{r_1(b)}{r_2(b)}$ and {so}
\beq \lab{eq:20240923-1613}
S(u_c)=S(b)=\frac{\omega_1(b)}{\omega_2(b)}
=\frac{r_{1}^{N-1}(b)|r'_2(b)|^{m-1}}{r_{2}^{N-1}(b)|r'_1(b)|^{m-1}}
=\left(\frac{r_1(b)}{r_2(b)}\right)^{N-m}=s(b)^{N-m}.
\eeq
{Since $s$ is strictly} decreasing in $(b,\alpha_1)$, we see that $s(b)>s(u_I)=1$ and {so}
\beq\lab{eq:20250109-1308}
S(u_c)=s(b)^{N-m}<s(b)^N=s^N(u_c).
\eeq
If \eqref{eq:20240923-1558} holds, then $b\leq u_c<\bar{u}_I<u_I<\alpha_1$. If $u_c=b$,
{Proposition~\ref{prop:20240923-1148} gives} that $t(u_c)>0$, {so} that $s(u_c)=s(b)>1$. In particular, $t'(b)<0$ and {so} $\dfrac{|r'_2(b)|}{|r'_1(b)|}<1$. {Hence,}
\beq\lab{eq:20240923-1649}\begin{aligned}
S(u_c)&=S(b)=\frac{r_{1}^{N-1}(b)|r'_2(b)|^{m-1}}{r_{2}^{N-1}(b)
|r'_1(b)|^{m-1}}<\left(\frac{r_1(b)}{r_2(b)}\right)^{N-1}
=s(b)^{N-1}\\
&<s(b)^N=s(u_c)^N.\end{aligned}
\eeq
If $u_c>b$, then $t'(u_c)=0$, $t(u_c)>0$ and  $t''(u_c)<0$.
{Hence}, $s(u_c)>1$ and {so} it also holds that
\beq\lab{eq:20240923-1655}
S(u_c)=\frac{r_{1}^{N-1}(u_c)|r'_2(u_c)|^{m-1}}{r_{2}^{N-1}
(u_c)|r'_1(u_c)|^{m-1}}
=\left(\frac{r_1(u_c)}{r_2(u_c)}\right)^{N-1}
=s^{N-1}(u_c)<s^N(u_c).
\eeq
 Put $h(u):=s(u)^N-S(u_c)$, $u\in [b, \alpha_1]$. {Then,} by \eqref{eq:20250109-1308}, \eqref{eq:20240923-1649} and \eqref{eq:20240923-1655}, we {derive} that $h(u_c)=s(u_c)^N-S(u_c)> 0$. On the other hand, $h(\alpha_1)=0-S(u_c)<0$. {Hence}, by the intermediate value theorem and monotonicity, there exists a unique $u_e\in [b,\alpha_1)$ such that $S(u_c)= s(u_e)^N$ and \eqref{eq:20240923-1601} holds.
\ep

{Now, we are ready to complete the proof of
Theorem~\ref{th:20250107-1127} for $2\leq N\leq m$ as follows.}
Put
\beq\lab{eq:20240921-1255}
a=\frac{N}{1+g(u_e)},
\eeq
{which leads to the relation $g(u_e)=\dfrac{N-a}{a}$.
By assumption \ref{hcon_3}, we have $g(u_e)>-1$, which ensures $a>0$. Note that \ref{hcon_3} is only used at the point $u_e$ in the argument.}
 We also note that $a+1-\dfrac{N}{m}\geq a>0$, {since $2\le N\leq m$.
Thus}, by
Proposition~\ref{prop:20240919-2000}, we deduce that
\begin{align*}
P_1(&u_c)-S(u_c)P_2(u_c)\\
&=\left(a+1-\frac{N}{m}\right)
\int_{\alpha_1}^{u_c}\omega_1(u)\mathrm{d}u +\frac{a}{N}\int_{\alpha_1}^{u_c}r_1(u)^N f(u)\left[g(u)-\frac{N-a}{a}\right]\mathrm{d}u\\
&\;\;-S(u_c)\left\{\left(a+1-\frac{N}{m}\right)
\int_{\alpha_2}^{u_c}\omega_2(u)\mathrm{d}u +\frac{a}{N}\int_{\alpha_2}^{u_c}r_2(u)^N f(u)\left[g(u)-\frac{N-a}{a}\right]\mathrm{d}u\right\}\\
&=\left(a+1-\frac{N}{m}\right)\int_{\alpha_1}^{u_c}\omega_2(u)[S(u)-S(u_c)]\mathrm{d}u
+\left(a+1-\frac{N}{m}\right)S(u_c)
\int_{\alpha_1}^{\alpha_2}\omega_2(u)\mathrm{d}u\\
&\;\;+\frac{a}{N}\int_{\alpha_1}^{u_c}f(u)
\big[g(u)-g(u_e)\big]r_2(u)^N\big[s(u)^N-S(u_c)\big]\mathrm{d}u\\
&\;\;+\frac{a}{N}{S(u_c)}\int_{\alpha_1}^{\alpha_2}
f(u)\big[g(u)-g(u_e)\big]r_2(u)^N\mathrm{d}u\\
&=:I_1+I_2+I_3+I_4.
\end{align*}
Lemma \ref{prop:20240923-1148}
gives that $t'(u)<0$ for $u\in (u_c,\alpha_1)$. {Hence,} combining {this}, with Lemma~\ref{lemma:20260123-2021}-\ref{cro:20240920-1249},
we obtain that {$S'<0$ in} $(u_c,\alpha_1)$ and {so}
${S}<S(u_c)$ in $(u_c,\alpha_1)$. Together with $\omega_2(u)<0$ and $u_c<\alpha_1$, we see that $I_1<0$. {Now,} $S(u_c)>0$, $\alpha_1<\alpha_2$ and $\omega_2<0$ in $(\alpha_1,\alpha_2)$ imply that $I_2<0$.

{
For $u \in (u_c, u_e)$, Lemma~\ref{prop:20240923-1600} implies
 $s(u)^N-S(u_c)>0$ and the monotonicity in \ref{hcon_2} yields $g(u) \geq g(u_e)$. While for $u\in (u_e,\alpha_1)$, Lemma~\ref{prop:20240923-1600} implies $s(u)^N-S(u_c)<0$ and the monotonicity in \ref{hcon_2} yields $g(u) \leq g(u_e)$.
Hence, we always have that
$$[g(u) - g(u_e)\big] \big[s(u)^N - S(u_c)\big]\geq 0,\quad u\in (u_c, \alpha_1).$$
By \ref{hcon_1}, we have $f(u) > 0$ on $(u_c, \alpha_1)$.
Therefore, the integrand satisfies
$$
f(u)\big[g(u) - g(u_e)\big] r_2(u)^N \big[s(u)^N - S(u_c)\big] \geq 0, \quad u \in (u_c, \alpha_1).
$$
Consequently,
$$
\int_{\alpha_1}^{u_c} f(u)\big[g(u) - g(u_e)\big] r_2(u)^N \big[s(u)^N - S(u_c)\big] \mathrm{d}u \leq 0.
$$
Since $\frac{a}{N}>0$, it follows that
$$
I_3:=\frac{a}{N}\int_{\alpha_1}^{u_c} f(u)\big[g(u) - g(u_e)\big] r_2(u)^N \big[s(u)^N - S(u_c)\big] \mathrm{d}u \leq 0.
$$
}

{Since $b<u_e<\alpha_1$,  condition \ref{hcon_2} gives that $g(u)-g(u_e)\leq 0$ and condition \ref{hcon_1}  implies $f(u)>0$ for $u\in (\alpha_1,\alpha_2)$. Therefore, the integrand satisfies
$$f(u)\big[g(u)-g(u_e)\big]r_2(u)^N\leq 0, u\in (\alpha_1,\alpha_2).$$
Consequently, combining with $S(u_c)>0$ and $\frac{a}{N}>0$, $I_4 \leq 0$  holds as well.}

{In summary, we obtain
$$
P_1(u_c) - S(u_c)P_2(u_c) = I_1 + I_2 + I_3 + I_4 < 0,
$$
contradicting \eqref{eq:20240923-1402} established in \textbf{Step 1} and completing the proof.}
\hfill$\Box$
\medskip

\section*{Acknowledgements}

{Xuexiu Zhong was supported by the NSFC (No.12271184), Young Top-notch Talent Project of Guangdong Province (2024TQ08A725), Guangdong Basic and Applied Basic Research Foundation (2021A1515010034), Guangzhou Basic and Applied Basic Research Foundation(2024A04J10001).}

{Jianjun Zhang was supported by the NSFC (No.12371109).}

P. Pucci is a member of the {\em Gruppo Nazionale per
l'Analisi Ma\-te\-ma\-ti\-ca, la Probabilit\`a e le loro Applicazioni} (GNAMPA) of the {\em Instituto Nazionale di Alta Matematica} (INdAM)
and this paper was written under the auspices of GNAMPA--INdAM.

This research was completed when P. Pucci was visiting the {Sichuan University}, Chengdu, under the invitation of Professor S. Zhang and supported by a grant from the {\it State Administration of Foreign Experts Affairs} (No. S20240071), titled {\it `` The Study for a few Important Nonlinear Equations"}.

{The authors are extremely grateful to  the anonymous reviewer for his/her valuable comments and suggestions, which improved the quality of the paper.}

\end{document}